\documentclass[9pt,twocolumn,twoside]{osajnl}
\journal{pr} 

\usepackage{amsthm,amssymb,amsmath}
\usepackage{bookmark}
\usepackage{placeins}
\usepackage{hyperref}
\usepackage{tikz}
\usepackage{pgfplots}
\pgfplotsset{compat=1.18}
\usepackage{graphicx}
\usepackage{svg}
\usepackage{subcaption}
\usepackage{amsmath,amssymb,amsthm,mathtools}
\usepackage{microtype}
\usepackage{enumitem}
\usepackage{hyperref}
\usepackage{xcolor}
\usepackage{graphicx}     
\usepackage{float}   
\usepackage{graphicx}
\usepackage{subcaption}   
\usepackage{float}        
\usepackage{placeins}     
\usepackage{subcaption}   
\usepackage{tikz}         
\graphicspath{{figs/}}    
\usepackage{tikz}
\usepackage{pgfplots}
\pgfplotsset{compat=1.18} 
\usepackage{svg}
\usepackage{needspace}
\usepackage{etoolbox}
\makeatletter
\newcommand{\@doi}{}
\makeatother
\AtBeginEnvironment{proof}{\Needspace{12\baselineskip}}

\newtheorem{theorem}{Theorem}[section]
\newtheorem{lemma}[theorem]{Lemma}
\newtheorem{proposition}[theorem]{Proposition}
\newtheorem{corollary}[theorem]{Corollary}
\newtheorem*{rem*}{Remark}

\setboolean{shortarticle}{false} 
\setboolean{displaycopyright}{false}

\title{Geometry of Deformed Cellular Spaces}

\author[1]{Shlomo Barak}
\author[2]{George Salman}

\affil[1]{Taga Innovations, 18 Yosef Karo St., Tel Aviv 6701422, Israel}

\affil[2]{School of Computer Science and Engineering, The Hebrew University of Jerusalem, Edmond J. Safra Campus, Jerusalem 9190401, Israel}

\affil[*]{Corresponding author: George.salman@mail.huji.ac.il}

\begin{abstract}
We present an \emph{Adaptive Geometry} in which the yardstick co-deforms with space itself—formulated on cellular spaces where \emph{length is a count}: distances are shortest cell-crossing counts. No cell shape, angles, or embedding are assumed; the framework is deliberately micro-agnostic. Curvature/deformation is inferred operationally by comparing a measured radius to the radius reconstructed from boundary/area/volume counts; the linear dimension of a cell serves as the universal, one and only one, unit of length, yielding unified small-ball/small-sphere estimators in 2D/3D/4D. We prove that the count metric on locally finite complexes is geodesic; show flatness on uniform lattices; and establish stability of distances and curvature estimators under small local perturbations. As a bridge to the smooth setting, a line-density field induces a conformal metric $g = e^{2u}g_{0}$ that recovers the same operational quantities. We outline a Ricci-like construction from directional slices and give a spherically symmetric illustrative example consistent with Schwarzschild spatial behavior metric. Altogether, our model supplies an intrinsic, micro-agnostic calculus linking discrete measurements to continuum notions with guarantees, including Gromov--Hausdorff (GH) control under mild regularity assumptions.
\end{abstract}

\begin{document}
\maketitle
\section{Introduction}
The concept of a geometry of deformed cellular spaces and its physical
implications was articulated earlier \cite{BarakHAL}; the present work provides its first rigorous mathematical formulation and analysis. Geometry has long served as the canonical language of physics. Euclidean geometry models homogeneous, isotropic continua, while Riemannian geometry extends this framework to curved manifolds. Since Einstein’s general relativity, Riemannian methods have underpinned quantitative descriptions of curvature, geodesics, and gravitational phenomena \cite{Einstein1933,Rindler2004}.

Notwithstanding these achievements, substantive conceptual and practical tensions persist. The intrinsic Riemannian apparatus presupposes differentiability that is ill-suited to granular or strongly inhomogeneous media \cite{Regge1961,BuragoBuragoIvanov2001}. For example, when length must be inferred from discrete traversals through a cellular substrate, angle and embedding data are neither available nor relevant. Alongside persistent problems in gravitation and cosmology—from near-flatness on cosmic scales to the organization of microstructure—these considerations motivate an operational program in which geometry is extracted from controlled measurements rather than imposed by coordinates \cite{Morgan1998,doCarmo1992}. Developing this program with rigorous estimators, stability bounds, and clear limiting links to the smooth theory is necessary and, in our view, urgent.

In this work, a geometry is formulated for $n$--dimensional spaces that admit a \emph{cellular lattice} without prescribing cell shape. The lattice is treated as \emph{elastic}: cells may contract or dilate and the lattice may vibrate. \emph{Curvature around a point} is diagnosed via the \emph{small--circle (small--ball) excess}: the difference between the measured radius and the radius reconstructed from boundary/area/volume counts, both taken with the same yardstick—the cell’s linear dimension (one crossing). This pointwise diagnostic accommodates \emph{local} contraction or dilation near sources while the \emph{global} space may remain Euclidean (zero global curvature).

Intrinsic measurements are taken with rulers and clocks that co--deform with space; ambient labels may be attached for visualization but play no role in the intrinsic constructions. The standing assumptions are as follows. (A1) \textbf{Cellularity}: space is a locally finite cell complex with face adjacency. (A2) \textbf{Elasticity and micro--agnosticism}: cell sizes may vary and shapes need not be specified; neither angles nor an embedding are assumed. (A3) \textbf{Intrinsic yardstick}: one cell crossing defines one unit of length, and distances are shortest crossing counts evaluated by a single, consistent protocol within any experiment. (A4) \textbf{Locality}: all statements and guarantees are local and require only finite subregions; global unboundedness is a background assumption, not a technical requirement. (A5) \textbf{Local near--uniformity of cell scale}: for every compact region $K$ there is a constant $\Lambda_K \ge 1$ such that whenever two cells $c,c'$ in $K$ share a face,
\[
\Lambda_K^{-1} \;\le\; \frac{\ell(c')}{\ell(c)} \;\le\; \Lambda_K
\]
Equivalently, the intrinsic yardstick $\ell$ (one crossing $=$ $\ell$ units of physical length) does not exhibit arbitrarily large discontinuous jumps between adjacent cells. This implies that in a sufficiently small neighborhood one can assign a well-defined local step length $\ell_i(x)$ in each principal direction (and, in isotropic patches, a single $\ell(x)=e^{u(x)}$), which is what we use in Sec.~\ref{subsec:line-element} and in the conformal relaxation $g = e^{2u} g_{0}$.

The framework is developed in parallel in discrete and smooth languages. Discretely, space is modeled as an elastic lattice of fundamental cells; distances are obtained by shortest \emph{cell--crossing counts}, and curvature is inferred from small--ball/sphere excess. In a smooth relaxation, a line--density (yardstick) field is introduced to induce a conformal metric $g=e^{2u}g_{0}$ that reproduces the same operational quantities; within this relaxation, the \emph{Ricci tensor} and \emph{Ricci scalar} coincide with their standard Riemannian expressions, while the interpretation remains measurement--first \cite{BobenkoSuris2008,Petersen2016}. Under spherical symmetry, the resulting measurement--based scalars are observed to reproduce the qualitative behavior of Schwarzschild spatial slices (as a mathematical cameo), and weighted shortest paths recover the expected bending regime \cite{Wald1984,MTW1973}. Emphasis is placed on geometric and operational aspects; physical interpretation is deferred to a companion work.

In this setting we prove several formal results. First, on any locally finite complex the cell–crossing (count) metric admits length minimizers, so the intrinsic space is geodesic (Sec.~\ref{sec:metric}). Second, in two dimensions the excess–radius diagnostic recovers Gaussian curvature via the small–circle law: $\delta r_{\mathrm{per}}(r)=\tfrac{K(x)}{6}r^{3}+o(r^{3})$ and $\delta r_{\mathrm{area}}(r)=\tfrac{K(x)}{24}r^{3}+o(r^{3})$ (Thm.~\ref{thm:small-circle}, Sec.~\ref{sec:delta-r}). Third, for $m\in\{2,3,4\}$ we introduce a unified small–ball/sphere estimator
$K_h^{(m)}(r)=\frac{6}{m}\,\frac{r^{m}-(r_c^{(m)})^{m}}{r^{m+2}}$ (Eq.~\eqref{eq:unified}, Sec.~\ref{sec:estimator}); it has the correct contraction/dilation sign, vanishes on uniform lattices, and—after Euclidean calibration—identifies the smooth scalar curvature to leading order (Thm.~\ref{thm:R-local}, Sec.~\ref{sec:local-global-interpretation}). Fourth, by restricting counts to thin tubes around geodesic two–slices we obtain directional (sectional) signals that converge to $K_g(e_i\wedge e_j)$ and assemble $\mathrm{Ric}$ and $R$ (Sec.~\ref{sec:ricci}, Thm.~\ref{thm:sec-ident-explicit}). Finally, under mild regularity and mesh assumptions we establish quantitative stability and convergence: the curvature estimator obeys $\big|\widehat{R}_m(x;r)-R_g(x)\big|\le C_1\frac{a}{r}+C_2 r$ (Thm.~\ref{thm:rates}), and the metric–measure spaces $(X_a,a\,d_h,\nu_a)$ converge in the measured Gromov–Hausdorff sense to $(\Omega,d_g,\mathrm{vol}_g)$ (Thm.~\ref{thm:mgh}). A spherically symmetric example illustrates qualitative agreement with Schwarzschild spatial behavior.

\paragraph*{Roadmap.}
Section~\ref{sec:metric} develops the intrinsic count metric and proves geodesic existence. 
Section~\ref{sec:mm-interface} states the metric–measure hypotheses and the measured GH framework. 
Section~\ref{sec:delta-r} formalizes the excess radius and the small–circle law. 
Section~\ref{sec:estimator} introduces the unified small–ball/sphere estimators, verifies flatness on uniform lattices, and gives quantitative rates. 
Section~\ref{sec:local-global-interpretation} explains the conformal relaxation $g=e^{2u}g_{0}$ and the $R$–normalized identification. 
Section~\ref{sec:ricci} constructs directional (sectional) estimators and assembles $\mathrm{Ric}$ and $R$. 
Section~\ref{sec:counts-curvature} expresses curvature directly in terms of the density field. 
The appendices collect calibrations, Voronoi examples, and the proof of the measured Gromov–Hausdorff limit.


\section{Intrinsic metric from counts and geodesics}

\paragraph*{Operational protocol (single yardstick).}
All intrinsic measurements in this paper follow the same protocol:
\begin{enumerate}[label=(\roman*),leftmargin=1.5em]
\item Fix a base cell $c_0$.
\item Declare one face crossing to have unit cost. This fixes a \emph{yardstick} and a \emph{gauge} for the experiment.
\item For an integer $r\ge 0$, define the intrinsic (count) ball and sphere
\[
B(r):=\{c:\ d_h(c_0,c)\le r\},\qquad
S(r):=\{c:\ d_h(c_0,c)=r\}
\]
We call $r$ the \emph{count radius}: it is literally the minimum number of face crossings from $c_0$.
\item Record the raw counts on $B(r)$ and $S(r)$:
\[
C_h(r):=|S(r)|,\quad
A_h(r):=|B(r)|
\]
and in higher dimensions their $3$D/$4$D analogues $V_h(r),W_h(r)$.
\item Use undeformed baseline growth constants $(\alpha_1,\beta_2,\beta_3,\beta_4)$ for the chosen adjacency (Sec.~\ref{app:counts}) to \emph{reconstruct} a radius $r_c$ from these counts, e.g.
\[
r_c^{(1)}=\frac{C_h(r)}{\alpha_1},\qquad
r_c^{(2)}=\sqrt{\frac{A_h(r)}{\beta_2}},\ \ldots
\]
and compare $r$ and $r_c$.
\end{enumerate}
No angles, embeddings, shapes, or coordinates enter in steps (i)--(v); they involve only cell adjacency and counting with one fixed cost per crossing.

\label{sec:metric}
We work on a locally finite cellular complex with face adjacency. A (piecewise) path is a sequence of adjacent cells
\[
\gamma:\ c_0\to c_1\to\cdots\to c_k
\]
and its \emph{length} is the number of face crossings, $L(\gamma)=k$. The \emph{intrinsic (count) distance} between cells $u,v$ is
\begin{equation}
d_h(u,v)\ :=\ \min_{\gamma:u\rightsquigarrow v}L(\gamma)
\end{equation}
namely the fewest crossings needed to connect $u$ to $v$. This defines an integer--valued metric.

 \begin{theorem}[Geodesics on locally finite complexes]
If the adjacency graph is locally finite, then for every $u,v$ there exists a path attaining $d_h(u,v)$. In particular $(\mathrm{Cells},d_h)$ is a geodesic metric space.
\end{theorem}

\begin{proof}
Let $G=(\mathcal{C},E)$ be the adjacency graph of cells; local finiteness means each $c\in\mathcal{C}$ has finite degree. Fix $u\in\mathcal{C}$. Consider the breadth–first search (BFS) layers $L_0=\{u\}$ and, inductively,
\[
L_{k+1} \ :=\ \{\, w\in\mathcal{C}\setminus (L_0\cup\cdots\cup L_k)\ :\ \exists\,z\in L_k\text{ with }(z,w)\in E \,\}
\]
By local finiteness, each $L_k$ is finite. If $v\in L_L$ for the first time at level $L$, then any path from $u$ to $v$ has length at least $L$, while the BFS tree provides a path $u\to\cdots\to v$ of length exactly $L$. Hence $d_h(u,v)=L$ and the minimizer exists. Therefore $(\mathcal{C},d_h)$ is geodesic. 
\end{proof}

\paragraph*{Notation and measurement primitives.}
Throughout, fix a base cell $c_0$ and write,
\[
B(r):=\{c:\ d_h(c_0,c)\le r\},\qquad
S(r):=\{c:\ d_h(c_0,c)=r\}
\]
so $S(r)$ is precisely the $r$th BFS layer and $B(r)$ the union of layers $0,\dots,r$. In $2$D we record the \emph{boundary} and \emph{area} counts
$C_h(r):=|S(r)|$ and $A_h(r):=|B(r)|$; in $3$D and $4$D we analogously use the \emph{volume} and \emph{hypervolume} counts, $V_h(r)$ and $W_h(r)$.

\paragraph*{Calibrations (one yardstick).}
To convert counts back into a “calculated” radius with the \emph{same} instrument, we compare against the undeformed growth of the chosen adjacency and fix once—and—for—all the constants
\[
C_0(r)\sim \alpha_1 r,\quad
A_0(r)\sim \beta_2 r^2,\quad
V_0(r)\sim \beta_3 r^3,\quad
W_0(r)\sim \beta_4 r^4
\]
These are not assumptions but reference values for the baseline lattice (exact examples are recorded in Appendix~\ref{app:counts}); they serve only to define $r_c$ from $C_h,A_h,V_h,W_h$ in the same gauge as the measured radius.

\medskip
\textbf{Gauge.} A global rescaling of the yardstick (e.g.\ measuring each crossing as $k$ units) rescales all lengths by $k$ but leaves dimensionless statements (signs of diagnostics, ratios, normalized densities) invariant. All comparisons below are made in a fixed gauge per experiment.

\subsection{Operational line element and emergent metric}
\label{subsec:line-element}

The metric is not postulated \emph{a priori}; it is manufactured from how many cells you cross. Fixing a point $x$ and restrict attention to a small neighborhood in which adjacent cells are (to first order) the same size, so that motion can be decomposed into locally orthogonal principal directions $i=1,2,3,\dots$ and each direction has a well-defined local linear size. Let $dn_i$ be the (integer) number of cell crossings along direction $i$ and let $\ell_i(x)$ be the local linear size of a single cell in that direction at $x$. Equivalently, the local line density in that direction is $\iota_i(x):=1/\ell_i(x)$.

The physical length increment assigned to such an infinitesimal move is then
\begin{equation}
ds^2 \;=\; \sum_i \big(\,\ell_i(x)\, dn_i\,\big)^2
\;=\; \sum_i \left(\frac{dn_i}{\iota_i(x)}\right)^2
\label{eq:ds2-anisotropic}
\end{equation}
In other words, in this local principal frame the metric tensor is diagonal,
\begin{equation}
g_{ij}(x)\;=\;\mathrm{diag}\!\big(\ell_1(x)^2,\ \ell_2(x)^2,\ \ell_3(x)^2,\dots\big),
ds^2 \;=\; g_{ij}(x)\, dn_i\, dn_j
\label{eq:gij-local}
\end{equation}

In an \emph{isotropic} patch the neighborhood is not only locally uniform in size but also directionally uniform, so all local cell lengths agree:
\begin{equation}
\ell_1(x)=\ell_2(x)=\cdots=\ell(x)
\end{equation}
Then \eqref{eq:ds2-anisotropic} simplifies to
\begin{equation}
ds \;=\; \ell(x)\,\sqrt{dn_1^2+dn_2^2+dn_3^2+\cdots},
\qquad
g_{ij}(x) \;=\; \ell(x)^2\,\delta_{ij}
\label{eq:isotropic-ds}
\end{equation}
Writing $\ell(x)=e^{u(x)}$ shows that this matches the conformal ansatz
\begin{equation}
g \;=\; e^{2u(x)}\,g_{0}
\end{equation}
used later in the smooth relaxation. Here $u(x)$ (equivalently $\iota(x)=e^{u(x)}$) is nothing more than the local yardstick: it is the linear size of \emph{one} cell in an almost-uniform neighborhood. Thus the smooth metric $g$ is not an additional assumption; it is just the continuum encoding of the operational rule
“count $dn_i$ cells along direction $i$ and multiply by that cell’s local linear size.”

\section{Metric–measure interface and hypotheses}
\label{sec:mm-interface}

Assume an open set $\Omega\subset\mathbb{R}^m$ and $u\in C^2(\Omega)$. Let $g=e^{2u}g_0$ with $g_0$ Euclidean.
For each mesh scale $a>0$, let $X_a$ be a locally finite cell complex with face adjacency and cell set $\mathcal{C}_a$.
A realization $\Phi_a:\mathcal{C}_a\to\Omega$ assigns e.g.\ cell barycenters.
Define $d_a^\ast:=a\,d_h$ and, for a cell-measure $\mu_a$ (counting or weighted), set
\[
\nu_a:=
\begin{cases}
a^m\,\mu_a,&\text{if $\mu_a$ counts cells},\\
\mu_a,&\text{if $\mu_a$ is already volumetric}.
\end{cases}
\]
Let $\omega_m$ be the Euclidean unit-ball volume and fix the calibration $\beta_m:=\omega_m$.

\paragraph*{Hypotheses on $K\Subset\Omega$ (constants independent of $a$).}
\begin{description}[leftmargin=1.6em,style=nextline]
\item[(H1) Local finiteness / degree bound:] each cell has $\le D$ face-adjacent neighbors.
\item[(H2) Shape regularity / mesh size:] $\mathrm{diam}_{g_0}(c)\in[c_0 a,C_0 a]$ whenever $\Phi_a(c)\in K$.
\item[(H3) Coarse realization:] $\Phi_a(\mathcal{C}_a\cap K)$ is $Ca$-dense, and $c\sim c'\Rightarrow d_{g_0}(\Phi_a(c),\Phi_a(c'))\le Ca$.
\item[(H4) Ball inclusions:] for all $x$ with $\Phi_a(x)\in K$ and admissible $r$,
\[
B_g(\Phi_a(x),\,ar - c_1 a)\subseteq \Phi_a(B_{d_h}(x,r)) \subseteq B_g(\Phi_a(x),\,ar + c_2 a).
\]
\item[(H5) Weight comparability:] there is $\Lambda\ge1$ with
$\Lambda^{-1} a^m e^{m u(\Phi_a(c))}\le \mu_a(\{c\})\le \Lambda a^m e^{m u(\Phi_a(c))}$ for $\Phi_a(c)\in K$.
\end{description}

\begin{theorem}[Measured GH limit]
\label{thm:mgh}
Under \textup{(H1)--(H5)}, as $a\to0$,
\[
\big(X_a,d_a^\ast,\nu_a\big)\ \xrightarrow[\mathrm{mGH}]{}\ \big(\Omega,d_g,\mathrm{vol}_g\big)
\]
\end{theorem}

\section{Excess radius \texorpdfstring{$\delta r$}{δr} and the small--circle law}
\label{sec:delta-r}
With the intrinsic metric $d_h$ and balls/spheres $B(r),S(r)$ from Sec.~\ref{sec:metric}, all measurements use a single yardstick: one face crossing has unit cost, and the same instrument is used to convert counts back into a ``calculated'' radius. See Fig.~\ref{fig:delta-r-schematic} for an illustration of the operational meaning of the excess radius.

At an intrinsic radius $r\in\mathbb{N}$ around a base cell $c_0$, form the reconstructed radii with the undeformed calibrations $(\alpha_1,\beta_2,\beta_3,\beta_4)$:
\[
r_c^{(1)}=\frac{C_h(r)}{\alpha_1},\qquad
r_c^{(2)}=\sqrt{\frac{A_h(r)}{\beta_2}},\qquad
r_c^{(3)}=\sqrt[3]{\frac{V_h(r)}{\beta_3}},\ \ldots
\]
The \emph{excess radius} compares two realizations of ``radius'' obtained with the \emph{same} instrument:
\begin{equation}
\delta r(r)\ :=\ r\;-\;r_c(r)
\label{eq:def-delta-r}
\end{equation}
where $r$ is the kinematic radius (graph distance) and $r_c$ is the combinatorial radius (counts inverted to radius).
The sign encodes geometry: if, at fixed $r$, the perimeter carries \emph{more} boundary cells than the undeformed baseline, then $r_c<r$ and $\delta r>0$ (\emph{contraction}); if it carries \emph{fewer}, then $r_c>r$ and $\delta r<0$ (\emph{dilation}).

In two dimensions the classical \emph{small–circle law} identifies the Gaussian curvature at the base point as the order–$r^3$ obstruction to reconciling the two radii:
\begin{equation}
K(x)\ =\ \lim_{r\downarrow 0}\ \frac{6\,\delta r(r)}{r^{3}}
\ =\ \frac{K(x)}{6}\,r^0\quad\Longleftrightarrow\quad
\delta r(r)=\frac{K(x)}{6}\,r^3+o(r^3)
\label{eq:small-circle}
\end{equation}
No angles, embeddings, or cell shapes are invoked. In the smooth relaxation with line–density $\iota=e^{u}$ and conformal metric $g=e^{2u}g_0$, the identity $K_g=-e^{-2u}\Delta u$ yields
\[
\delta r(r)\ =\ -\,\frac{e^{-2u(x)}}{6}\,\Delta u(x)\,r^{3} + o(r^{3}),
\]
so a local \emph{increase of the yardstick} (convex $u$) produces $\delta r>0$ and a positive signal in~\eqref{eq:small-circle}, consistent with the discrete reading.

\begin{figure}[t]
\centering
\begin{tikzpicture}[scale=0.92]
  \draw[line width=0.9pt] (0,0) circle (2.3);
  \draw[->] (0,0) -- (2.3,0) node[midway,below] {$r$};
  \draw[dashed,line width=0.9pt] (0,0) circle (1.8);
  \draw[->,dashed] (0,0) -- (1.8,0) node[midway,above] {$r_c$};
  \foreach \ang in {0,12,...,348} {\fill ({2.3*cos(\ang)},{2.3*sin(\ang)}) circle (0.7pt);}
  \foreach \ang in {6,30,60,96,126,156,186,216,246,276,306,336} {\fill ({2.3*cos(\ang)},{2.3*sin(\ang)}) circle (0.7pt);}
  \node[align=center] at (0,-2.8) {\footnotesize Contraction: increased perimeter count at fixed $r$ \\[-1pt]\footnotesize $\Rightarrow\ r_c<r$ and $\delta r>0$.};
\end{tikzpicture}
\caption{Operational meaning of $\delta r$: measured radius $r$ vs.\ reconstructed radius $r_c$ using the same yardstick.}
\label{fig:delta-r-schematic}
\end{figure}
\begin{theorem}[Small-circle law in 2D]
\label{thm:small-circle}
Let $g$ be $C^{2}$ near $x$. Then, as $r\downarrow 0$:

\noindent\textbf{Perimeter gauge.}
\begin{equation}
\begin{aligned}
L_g\!\big(\partial B_g(x,r)\big)
  &= 2\pi r\Big(1-\tfrac{K(x)}{6}\,r^{2}+o(r^{2})\Big),\\
r_{c,\mathrm{per}} &:= \frac{L_g}{2\pi},\qquad
\delta r_{\mathrm{per}} = r - r_{c,\mathrm{per}}
  = \tfrac{K(x)}{6}\,r^{3}+o(r^{3}).
\end{aligned}
\end{equation}

\noindent\textbf{Area gauge.}
\begin{equation}
\begin{aligned}
\operatorname{Area}_g\!\big(B_g(x,r)\big)
  &= \pi r^{2}\Big(1-\tfrac{K(x)}{12}\,r^{2}+o(r^{2})\Big),\\
r_{c,\mathrm{area}} &:= \sqrt{\frac{\operatorname{Area}_g}{\pi}},\qquad
\delta r_{\mathrm{area}} = \tfrac{K(x)}{24}\,r^{3}+o(r^{3})
\end{aligned}
\end{equation}

In a conformal chart $g=e^{2u}g_{0}$, these become
\begin{equation}
\begin{aligned}
\delta r_{\mathrm{per}}   &= -\tfrac{e^{-2u}}{6}\,\Delta u\,r^{3}+o(r^{3}),\\
\delta r_{\mathrm{area}}  &= -\tfrac{e^{-2u}}{24}\,\Delta u\,r^{3}+o(r^{3}).
\end{aligned}
\end{equation}
\end{theorem}

\begin{proof}[Proof.]
Let $\exp_x$ be the exponential map and use geodesic polar coordinates $(r,\theta)$.
The metric takes the form
\begin{equation}\label{eq:polar-metric}
g \;=\; dr^{2} \;+\; J(r,\theta)^{2}\,d\theta^{2}
\end{equation}
where $J$ is the length of the orthogonal Jacobi field along the unit-speed
geodesic $\gamma_\theta$ from $x$. It satisfies
\begin{align}
\partial_{rr} J(r,\theta) \;+\; K(\gamma_\theta(r))\,J(r,\theta) &= 0, \label{eq:jacobi}\\
J(0,\theta) = 0, \qquad \partial_r J(0,\theta) &= 1 \label{eq:jacobi-ic}
\end{align}
Since $K$ is continuous near $x$,
\begin{equation}\label{eq:K-along}
K(\gamma_\theta(r)) \;=\; K(x) \;+\; O(r) \qquad (r\downarrow 0)\ \ \text{uniformly in }\theta .
\end{equation}
A Taylor expansion of \eqref{eq:jacobi}–\eqref{eq:jacobi-ic} at $r=0$ gives
\begin{equation}\label{eq:J-exp}
J(r,\theta) \;=\; r \;-\; \tfrac{K(x)}{6}\,r^{3} \;+\; o(r^{3}) \qquad (r\downarrow 0)
\end{equation}
again uniformly in $\theta$.

\paragraph{Perimeter.}
From \eqref{eq:polar-metric} and \eqref{eq:J-exp},
\begin{equation}\label{eq:length}
\begin{aligned}
L_g\big(\partial B_g(x,r)\big)
&= \int_{0}^{2\pi} J(r,\theta)\,d\theta \\
&= 2\pi r \Big(1 - \tfrac{K(x)}{6}\,r^{2} + o(r^{2})\Big)
\end{aligned}
\end{equation}
Hence,
\begin{equation}\label{eq:rc-per}
r_{c,\mathrm{per}}
\;:=\; \frac{L_g}{2\pi}
\;=\; r \Big(1 - \tfrac{K(x)}{6}\,r^{2} + o(r^{2})\Big)
\end{equation}
and therefore,
\begin{equation}\label{eq:delta-per}
\delta r_{\mathrm{per}}
\;=\; r - r_{c,\mathrm{per}}
\;=\; \tfrac{K(x)}{6}\,r^{3} + o(r^{3})
\end{equation}

\paragraph{Area.}
Using $dA = J(r,\theta)\,dr\,d\theta$ and \eqref{eq:J-exp},
\begin{equation}\label{eq:area}
\begin{aligned}
\operatorname{Area}_g\!\big(B_g(x,r)\big)
&= \int_{0}^{r}\!\!\int_{0}^{2\pi} J(\rho,\theta)\,d\theta\,d\rho \\
&= 2\pi \int_{0}^{r}\!\Big(\rho - \tfrac{K(x)}{6}\rho^{3} + o(\rho^{3})\Big)\,d\rho \\
&= \pi r^{2} \Big(1 - \tfrac{K(x)}{12}\,r^{2} + o(r^{2})\Big)
\end{aligned}
\end{equation}
Thus,
\begin{equation}\label{eq:rc-area}
r_{c,\mathrm{area}}
\;:=\; \sqrt{\frac{\operatorname{Area}_g}{\pi}}
\;=\; r \Big(1 - \tfrac{K(x)}{24}\,r^{2} + o(r^{2})\Big)
\end{equation}
and,
\begin{equation}\label{eq:delta-area}
\delta r_{\mathrm{area}}
\;=\; r - r_{c,\mathrm{area}}
\;=\; \tfrac{K(x)}{24}\,r^{3} + o(r^{3})
\end{equation}

\paragraph{Conformal chart.}
If $g=e^{2u}g_{0}$ in $2$D, then
\begin{equation}\label{eq:Kg-conformal}
K(x) \;=\; -\,e^{-2u(x)}\,\Delta u(x).
\end{equation}
Substituting \eqref{eq:Kg-conformal} into \eqref{eq:delta-per}–\eqref{eq:delta-area} yields
\begin{align}
\delta r_{\mathrm{per}}
&= -\,\tfrac{e^{-2u(x)}}{6}\,\Delta u(x)\,r^{3} + o(r^{3}) \label{eq:delta-per-conf}\\
\delta r_{\mathrm{area}}
&= -\,\tfrac{e^{-2u(x)}}{24}\,\Delta u(x)\,r^{3} + o(r^{3})\label{eq:delta-area-conf}
\end{align}
\end{proof}

\paragraph*{Granularity, averaging, and gauge.}
Pointwise fields in a granular setting are understood as local averages over intrinsic balls:
\(
\langle f\rangle_r(x)=|B_x(r)|^{-1}\sum_{c\in B_x(r)}f(c).
\)
A global rescaling of the yardstick (counts $\mapsto k\cdot$counts, or $\iota\mapsto c\,\iota$) rescales lengths but leaves dimensionless quantities (signs of $\delta r$, ratios $r/r_c$, normalized densities) invariant. Continuous small–ball formulae are asymptotic and accurate when $r$ is large relative to cell size and the lattice is near–uniform at that scale; where exact counts are available, they take precedence.
\FloatBarrier

\section{Unified curvature estimator from counts}
\label{sec:estimator}
The unified curvature estimator, which connects excess-radius counts with curvature, is formulated below. Its geometric interpretation is illustrated in Fig.~\ref{fig:excess}. Let $m\in\{1,2,3,4\}$ denote the ambient dimension and let $r_c^{(m)}$ be the reconstructed radius obtained from the corresponding count with the \emph{same} yardstick as in Sec.~\ref{sec:metric} (e.g.\ $r_c^{(2)}=\sqrt{A_h/\beta_2}$). Define the unified small–ball/sphere estimator
\begin{equation}
\label{eq:unified}
K_h^{(m)}(r)\ :=\ \frac{6}{m}\,\frac{r^{m}-\bigl(r_c^{(m)}\bigr)^{m}}{r^{m+2}},
\qquad r\in\mathbb{N}
\end{equation}

\paragraph*{Leading connection to the excess radius.}
Write $w=r-r_c^{(m)}$. A binomial expansion gives
\begin{equation}
\label{eq:leading-delta}
K_h^{(m)}(r)\ =\ 6\,\frac{w}{r^{3}}\ +\ O\!\left(\frac{w^2}{r^4}\right)
\end{equation}
In particular, in $2$D and for symmetric neighborhoods, $K_h^{(2)}(r)=3\bigl(r^2-(r_c^{(2)})^2\bigr)/r^4\approx 6\,\delta r/r^3$, recovering the small–circle law of Sec.~\ref{sec:delta-r}.

\paragraph*{Sign and geometry (operational).}
At fixed $r$, if the boundary carries \emph{more} cells than the undeformed baseline, then $r_c<r$, so $w=\delta r>0$ and $K_h^{(m)}(r)>0$ (contraction). If it carries \emph{fewer}, then $w<0$ and $K_h^{(m)}(r)<0$ (dilation). Figure~\ref{fig:excess} illustrates both regimes.



\begin{figure}[t] 
  \centering
  \begin{subcaptionblock}{0.7\linewidth}
    \centering
    \includegraphics[width=\linewidth]{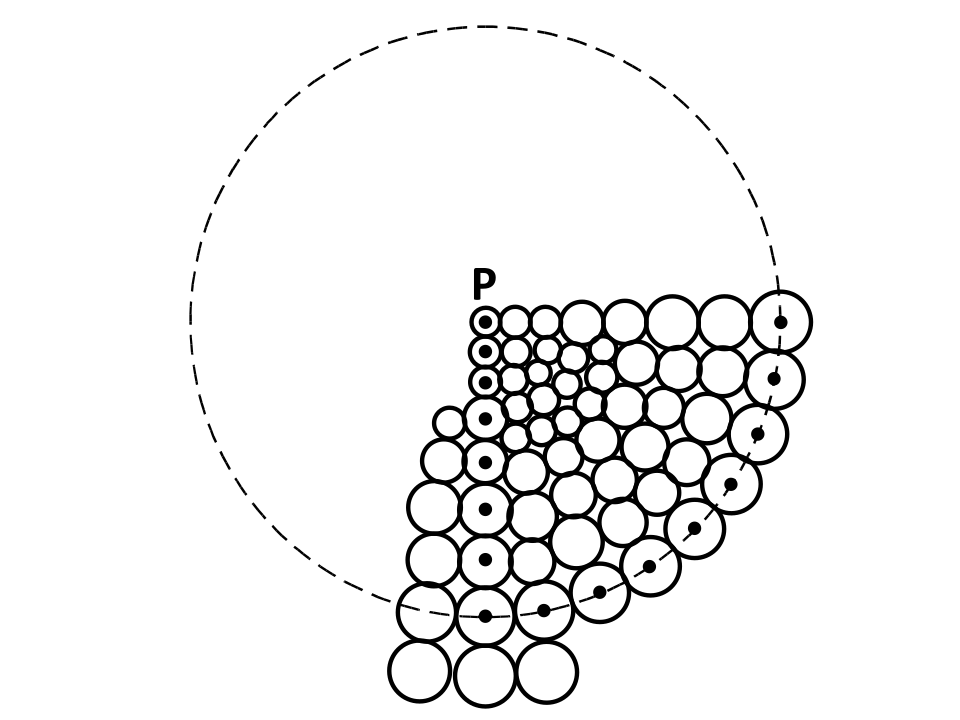}
    \subcaption{Positive curvature (contraction): \(\delta r>0\), \(K>0\).}
    \label{fig:excess-pos}
  \end{subcaptionblock}\hfill
  \begin{subcaptionblock}{0.7\linewidth}
    \centering
    \includegraphics[width=\linewidth]{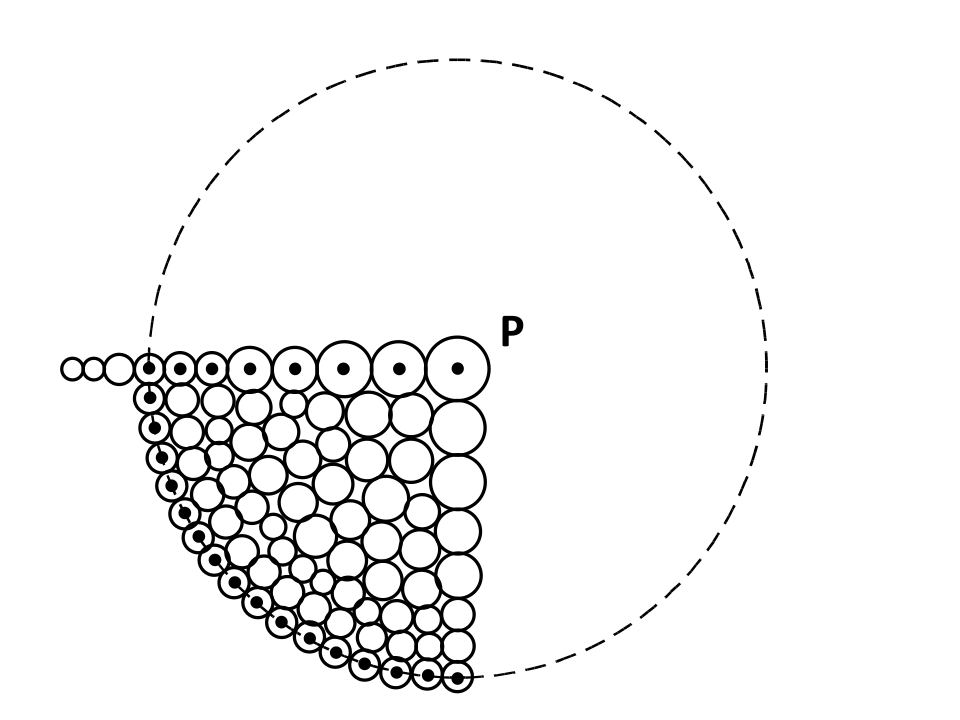}
    \subcaption{Negative curvature (dilation): \(\delta r<0\), \(K<0\).}
    \label{fig:excess-neg}
  \end{subcaptionblock}
  \caption{Excess radius \(\delta r=r-r_c\) at count–radius \(r\).
  Each panel shows a wedge of the intrinsic ball; dots mark outward boundary points per cell. The dashed circle is the measured radius \(r\); the reconstructed \(r_c\) is inferred from the boundary/area count using the \emph{same} yardstick.}
  \label{fig:excess}
\end{figure}

\paragraph*{Flatness on undeformed lattices.}
On standard uniform lattices (square/hex in 2D; cubic/hypercubic in higher $m$) one has $r_c^{(1)}=r$ exactly, hence $K_h^{(1)}\equiv 0$, and for $m\ge 2$
\[
K_h^{(m)}(r)=O(r^{-3})\ \longrightarrow\ 0\qquad (r\to\infty)
\]
so the estimator correctly registers flatness (exact formulas are tabulated in Appendix~\ref{app:counts}).

\paragraph*{Interpretation via line/space density.}
In the smooth relaxation, a \emph{line–density} (yardstick) field $\iota=e^{u}$ induces the conformal metric $g=e^{2u}g_0$. In dimension $m$,
\begin{equation}
\label{eq:R-conformal}
R_g \;=\; e^{-2u}\Big(-2(m-1)\Delta u\ -\ (m-1)(m-2)\lvert\nabla u\rvert^2\Big)
\end{equation}
so the scalar curvature at a point depends on the \emph{local value} of the density (via $e^{-2u}$) and its \emph{first/second derivatives}. For near–uniform shapes, the \emph{space density} satisfies $\rho\propto \iota^{m}$ (in $3$D: $\rho\propto \iota^{3}$), hence $u=\tfrac{1}{m}\log\rho$ and $R_g$ is determined by $\rho$ and its derivatives. Matching small–ball laws yields, for radii $r$ small compared with the curvature scale,
\begin{equation}
\label{eq:R-normalized}
K_{h,R}^{(m)}(r):=6(m+2)\,\frac{r^{m}-\bigl(r_c^{(m)}\bigr)^{m}}{r^{m+2}}
\ =\ R_g(x)\ +\ O(r)
\end{equation}
so to leading order the unified count estimator equals the smooth scalar curvature signal computed from the density field.


\begin{table}[t]
\centering
\caption{Small-ball coefficients and normalizations.}
\setlength{\tabcolsep}{6pt}        
\renewcommand{\arraystretch}{1.10} 
\small                              
\begin{tabular}{@{}lcc@{}}
\toprule
Dimension $m$ & $\mathrm{Vol}_g(B_r)=\omega_m r^m\big[1-\frac{R}{6(m+2)}r^2+\cdots\big]$ & $6(m{+}2)$ \\
\midrule
2 & $\omega_2 r^2\!\left(1-\frac{R}{24}r^2+\cdots\right)$ & 24 \\
3 & $\omega_3 r^3\!\left(1-\frac{R}{30}r^2+\cdots\right)$ & 30 \\
4 & $\omega_4 r^4\!\left(1-\frac{R}{36}r^2+\cdots\right)$ & 36 \\
\midrule
2D disk & $\mathrm{Area}(D_r)=\pi r^2\!\left(1-\frac{K}{12}r^2+\cdots\right)$ & 12 \\
\bottomrule
\end{tabular}
\end{table}

\medskip
Equations~\eqref{eq:unified}–\eqref{eq:R-normalized} put the operational picture on precise footing: \emph{curvature around a point} is the small–ball mismatch between measured and reconstructed radii with a single yardstick, and in the conformal relaxation it is governed by the local density and its derivatives through~\eqref{eq:R-conformal}.

\paragraph*{Admissible radii.}
Fix a compact $K\Subset\Omega$ and set
\[
r_{\max}(x,a)\ :=\ \left\lfloor \frac{\mathrm{dist}_g(\Phi_a(x),\,\partial K)-C a}{a}\right\rfloor_+ 
\]
A radius $r\in\mathbb{N}$ is \emph{admissible} if $1\le r\le r_{\max}(x,a)$
\begin{theorem}[Small-ball identification with rates]
\label{thm:rates}
Assume \textup{(H1)--(H5)} on $K\Subset\Omega$. Then there exist $C_1,C_2>0$
(depending only on bounds for $u,\nabla u,\nabla^2 u$ on $K$ and the shape/degree constants) such that
for all $x$ with $\Phi_a(x)\in K$ and all admissible $r$,
\[
\big|\widehat{R}_m(x;r)-R_g(\Phi_a(x))\big|\ \le\ C_1\,\frac{a}{r}\ +\ C_2\, r
\]
In particular, if $a\to0$, $r\to0$, and $a/r\to0$, then $\widehat{R}_m(x;r)\to R_g(\Phi_a(x))$ uniformly on $K$.
\end{theorem}

\begin{proof}
Fix $x$ with $\Phi_a(x)\in K$ and write $y:=\Phi_a(x)$. Let
\begin{equation}\label{eq:s-rho}
s \;:=\; \big\lfloor r/a \big\rfloor, 
\qquad \rho \;:=\; a s,
\qquad |\,r-\rho\,|\;\le\; a 
\end{equation}
Set $E_s:=\Phi_a\!\big(B_{d_h}(x,s)\big)$. By \textup{(H4)} there are $c_1,c_2>0$ (uniform on $K$) such that
\begin{equation}\label{eq:ball-inclusions}
B_g\!\big(y,\rho-c_1 a\big)\ \subset\ E_s\ \subset\ B_g\!\big(y,\rho+c_2 a\big)
\end{equation}

\paragraph{Volume expansion and quadrature error.}
On $K$, the small–ball expansion is uniform:
\begin{equation}\label{eq:vol-expansion}
\mathrm{Vol}_g\!\big(B_g(y,t)\big)
\;=\; \alpha_m t^m \;-\; \alpha_m \frac{R_g(y)}{6(m+2)}\,t^{m+2} \;+\; O_K(t^{m+3}) ,
\end{equation}
with $\alpha_m=\omega_m$ and the $O_K(\cdot)$ depending only on bounds of $u,\nabla u,\nabla^2 u$ on $K$.
Let $\nu_a$ be the measurement measure (counting scaled by $a^m$ or volumetric, cf. (H5)).
A cellwise Taylor estimate together with \textup{(H2)}–\textup{(H5)} yields the quadrature bound
\begin{equation}\label{eq:quad}
\Big|\,\nu_a\!\big(B_{d_h}(x,s)\big) \;-\; \mathrm{Vol}_g(E_s)\,\Big|
\;\le\; C_K\, a\, \rho^{\,m} 
\end{equation}
for some $C_K>0$ depending only on shape/degree constants and $C^2$ bounds of $u$ on $K$.

\paragraph{Bounding the reconstructed radius.}
Define $r_c^{(m)}$ by
\begin{equation}\label{eq:rc-def}
\big(r_c^{(m)}\big)^{m} \;:=\; \frac{\nu_a\!\big(B_{d_h}(x,s)\big)}{\alpha_m}
\end{equation}
From \eqref{eq:ball-inclusions}–\eqref{eq:vol-expansion}–\eqref{eq:quad} and a first–order expansion of
$t\mapsto t^{m}$ and $t\mapsto t^{m+2}$ around $t=\rho$ we obtain
\begin{equation}\label{eq:rc-expansion}
\begin{aligned}
\big(r_c^{(m)}\big)^{m}
&= \rho^{m} - \frac{R_g(y)}{6(m+2)}\,\rho^{m+2} + E_1(\rho,a)\\
|E_1(\rho,a)|
&\le C_K\big(a\,\rho^{m-1} + \rho^{m+3}\big).
\end{aligned}
\end{equation}

Moreover, by \eqref{eq:s-rho},
\begin{equation}\label{eq:r-to-rho}
r^{m} \;=\; \rho^{m} \,+\, O\!\big(a\,r^{m-1}\big),
\qquad
r^{m+2} \;=\; \rho^{m+2} \,+\, O\!\big(a\,r^{m+1}\big),
\end{equation}
with constants uniform on $K$.

\paragraph{Main difference bound.}
Combine \eqref{eq:rc-expansion}–\eqref{eq:r-to-rho} to get
\begin{equation}\label{eq:Delta}
\begin{aligned}
\Delta
&:=\ r^{m} \;-\; \big(r_c^{(m)}\big)^{m}
      \;-\; \frac{R_g(y)}{6(m+2)}\,r^{m+2} \\
&=\ \Big[\rho^{m} \;-\; \big(r_c^{(m)}\big)^{m}
      \;-\; \frac{R_g(y)}{6(m+2)}\,\rho^{m+2}\Big]
\\
&\quad
+\; O\!\big(a\,r^{m-1}\big)
  \;+\; O\!\big(a\,r^{m+1}\big)
\\
&=\ -\,E_1(\rho,a)+\; O\!\big(a\,r^{m-1}\big)
  \;+\; O\!\big(a\,r^{m+1}\big)
\end{aligned}
\end{equation}

hence, using \eqref{eq:rc-expansion},
\begin{equation}\label{eq:Delta-bound}
|\Delta|\ \le\ C_K\big(a\,r^{m-1} \,+\, r^{m+3}\big).
\end{equation}

\paragraph{Conclusion.}
Recall the normalized estimator
\begin{equation}\label{eq:Rhat-def}
\widehat{R}_m(x;r)
\;:=\; 6(m+2)\,\frac{r^{m}-\big(r_c^{(m)}\big)^{m}}{r^{m+2}}
\end{equation}
By \eqref{eq:Delta-bound},
\begin{equation}\label{eq:final-rate}
\big|\widehat{R}_m(x;r) - R_g(y)\big|
\;=\; 6(m+2)\,\frac{|\Delta|}{r^{m+2}}
\;\le\; C_1\,\frac{a}{r} \;+\; C_2\, r 
\end{equation}
with $C_1,C_2$ depending only on shape/degree constants and $C^2$ bounds of $u$ on $K$.
This bound is uniform for $x$ with $\Phi_a(x)\in K$. In particular, if $a\to 0$, $r\to 0$, and $a/r\to 0$,
then $\widehat{R}_m(x;r)\to R_g(\Phi_a(x))$ uniformly on $K$.
\end{proof}


\section{Local versus global curvature; density interpretation and interface with smooth geometry}
\label{sec:local-global-interpretation}

Curvature in the present framework is extracted \emph{locally}, from measurements made with a single yardstick on a small intrinsic ball. The diagnostic depends only on data inside that ball and is insensitive to remote structure. At the same time, one may work in a global gauge in which the background is Euclidean while allowing localized deviations (contractions/dilations) near sources. The purpose of this section is to formalize this local/global split and its reading through a smooth conformal relaxation.

A tempting alternative is to “just count cells’’ globally and infer curvature from growth rates. This is gauge–unsafe: a uniform rescaling of the yardstick (or a homogeneous change of cell density) alters all counts by the same factor and can mimic or mask curvature. The excess–radius protocol avoids this pitfall by \emph{comparing two radii measured with the same instrument}: the kinematic radius $r$ and the reconstructed radius $r_c$ obtained from boundary/area/volume counts in the \emph{same gauge}. Only their mismatch is geometrically meaningful. 

\paragraph*{Conformal relaxation and density field.}
We interpret the discrete, measurement–first layer through a smooth \emph{conformal relaxation}. A positive \emph{line–density} (yardstick) field $\iota=e^{u}$ on an open set $U\subset\mathbb{R}^{m}$ induces
\begin{equation}
g\ :=\ e^{2u}\,g_0,
\qquad
R_g\ =\ e^{-2u}\Big(-2(m\!-\!1)\,\Delta u-(m\!-\!1)(m\!-\!2)\,|\nabla u|^2\Big)
\label{eq:R-conformal-sec4}
\end{equation}
All curvature tensors are those of the ordinary Riemannian metric $g$; the discrete counts provide operational access to their \emph{local} leading signals.

\paragraph*{$R$–normalized estimator.}
For $m\in\{2,3,4\}$ let $r\in\mathbb{N}$ be the intrinsic count–radius about $x$ and let $r_c^{(m)}(x;r)$ be the radius reconstructed from the $m$–dimensional count (Sec.~\ref{sec:metric}). Fix baseline constants $\beta_m$ so that in the undeformed reference $\mathrm{Vol}_{g_0}(B_{g_0}(x,r))=\alpha_m r^m$ with $\beta_m=\alpha_m$ (exact values in App.~\ref{app:counts}). Define
\begin{equation}
\label{eq:Rnorm-sec4}
K_{h,R}^{(m)}(x;r)\ :=\ 6(m\!+\!2)\,\frac{r^{m}-\bigl(r_c^{(m)}(x;r)\bigr)^{m}}{r^{m+2}}
\end{equation}

\begin{theorem}[Local identification with scalar curvature]
\label{thm:R-local}
Assume $g=e^{2u}g_0$ with $u\in C^{2}$ near $x$. Using \eqref{eq:Rnorm-sec4} with $\beta_m=\alpha_m$,
\[
\lim_{r\downarrow 0}\ K_{h,R}^{(m)}(x;r)\ =\ R_g(x)
\]
\emph{Proof.}
For a smooth Riemannian metric,
\[
\mathrm{Vol}_g\big(B_g(x,r)\big)=\alpha_m r^m-\alpha_m\,\frac{R_g(x)}{6(m+2)}\,r^{m+2}+O(r^{m+3})
\]
By definition, $\bigl(r_c^{(m)}\bigr)^m=\mathrm{Vol}_g(B_g(x,r))/\beta_m$. With $\beta_m=\alpha_m$,
\[
r^{m}-\bigl(r_c^{(m)}\bigr)^{m}
=\frac{R_g(x)}{6(m+2)}\,r^{m+2}+O(r^{m+3})
\]
Multiplying by $6(m+2)/r^{m+2}$ gives $K_{h,R}^{(m)}(x;r)=R_g(x)+O(r)$.\qed
\end{theorem}

\begin{proposition}[Discrete-to-smooth consistency]
\label{prop:discrete-consistency}
Let the cell complex be locally finite with mesh size $a>0$ and bounded aspect ratio. Suppose the count ball $B_h(x,r)$ is sandwiched by metric balls $B_g(x,r-c_1 a)\subseteq B_h(x,r)\subseteq B_g(x,r+c_2 a)$ and that the per-cell $g$–volume is uniformly bounded above/below. Then, as $a\to0$ with $r\to0$ and $a/r\to0$,
\[
K_{h,R}^{(m)}(x;r)\ =\ R_g(x)\ +\ o(1).
\]
\emph{Proof.}
The inclusions give $\mathrm{Vol}_g(B_g(x,r-c_1 a))\le |B_h(x,r)|\le \mathrm{Vol}_g(B_g(x,r+c_2 a))$. Apply the small–ball expansion at $r\pm c a$ with $\beta_m=\alpha_m$ and expand $r^{m}-\bigl(r_c^{(m)}\bigr)^{m}$; scaling by $6(m+2)/r^{m+2}$ yields an error $O(a/r)+O(r)$, which vanishes under $a/r\to 0$ and $r\to 0$. \qed
\end{proposition}

\begin{proposition}[Local/global split in a fixed gauge]
\label{prop:global-flat}
If $u$ is constant on an open region $\Omega$ (so $g=e^{2u_0}g_0$ there), then for every compact $K\Subset\Omega$,
\[
\lim_{r\downarrow 0}\ \sup_{x\in K}\,|K_{h,R}^{(m)}(x;r)|\ =\ 0.
\]
\emph{Proof.}
If $u$ is constant, then $R_g\equiv 0$ on $\Omega$. Apply Theorem~\ref{thm:R-local} uniformly on $K$. \qed
\end{proposition}

\paragraph*{Manifold status and the discrete layer.}
The measurement layer is a locally finite cell complex with the count metric $d_h$; no manifold hypothesis is required to define lengths, balls/spheres, or the estimators. When cells are quasi-uniform and links are manifold-like, $(B_R,d_h)$ and $(B_R,d_g)$ are $(1+\delta)$ quasi-isometric with $\delta=O(\mathrm{osc}_{B_R}u)$; under refinement and controlled oscillation they converge in the Gromov–Hausdorff sense. Curvature is thus extracted operationally from counts and read geometrically via $g=e^{2u}g_0$.

\paragraph*{Ricci and scalar curvature (compatibility).}
Because the relaxation is the standard metric $g=e^{2u}g_0$, the Ricci tensor $\mathrm{Ric}(g)$ and scalar $R_g$ are exactly those of Riemannian geometry. Section~\ref{sec:estimator} provides count-based local scalars consistent with $R_g$ by Theorem~\ref{thm:R-local}; directional slicing (Sec.~\ref{sec:ricci}) assembles a discrete Ricci-like object that converges to $\mathrm{Ric}(g)$ under the same regularity/refinement assumptions.

\section{Curvature tensor from directional slices}
\label{sec:ricci}

Curvature in our framework is fundamentally a \emph{2D diagnostic}: on any small intrinsic disk you can run the same ``one yardstick, two radii'' protocol from Sec.~\ref{sec:delta-r} and get a curvature signal. In higher dimensions, sectional, Ricci, and scalar curvature are all assembled from those 2D signals. This section makes that precise and shows (i) how to get a directional/sectional curvature estimate by \emph{counting only cells}, and (ii) how to combine those direction-by-direction signals into $\mathrm{Ric}$ and $R$ with rates and convergence guarantees. The corresponding geometric configuration is illustrated schematically in Fig.~\ref{fig:tube-figure}.
.
\paragraph*{Setup.}
We continue under the metric--measure assumptions (H1)--(H5) on a compact set $K\Subset\Omega$ from Sec.~\ref{sec:mm-interface}. The discrete space at mesh size $a>0$ is a locally finite cell complex $X_a$ with the intrinsic count metric $d_h$; $\Phi_a:X_a\to\Omega$ is the realization map into the smooth relaxation $(\Omega,g)$, where $g=e^{2u}g_0$. The per-cell measure $\nu_a$ is comparable to $g$-volume (H5). All measurements below still use a \emph{single} yardstick: one face crossing has unit cost.

At a base cell $x\in X_a$ (with $\Phi_a(x)\in K$), pick a $g$-orthonormal frame 
$\{e_1,\dots,e_m\}$ in $T_{\Phi_a(x)}\Omega$.  
For each $i<j$, let
\[
\Sigma_{ij}\ :=\ 
\exp_{\Phi_a(x)}\!\big(\mathrm{span}\{e_i,e_j\}\big)
\]
be the $2$--dimensional geodesic surface spanned by $e_i,e_j$.  Think of $\Sigma_{ij}$ as the infinitesimal $ij$-plane through $\Phi_a(x)$, but living inside $(\Omega,g)$.

\medskip
\noindent\textbf{Step 1. A purely 2D curvature signal on the slice.}

Inside $\Sigma_{ij}$ we can repeat the $2$D excess-radius protocol from Sec.~\ref{sec:delta-r}: take a small $g$-geodesic disk of radius $\rho$, compute its $g$--area, reconstruct a radius from that area using the same 2D calibration $\beta_2=\pi$, and compare.

Define
\begin{equation}
\label{eq:Kgeom-def}
\widehat{K}^{\mathrm{geom}}_{ij}(\Phi_a(x);\rho)
\ :=\
3\,\frac{
  \rho^{2}
  - \bigl(r^{(2)}_{c,ij}(\rho)\bigr)^{2}
}{
  \rho^{4}
}
\end{equation}
where
\[
\bigl(r^{(2)}_{c,ij}(\rho)\bigr)^2
\ :=\
\frac{
  \operatorname{Area}_g\!\big(B_{\Sigma_{ij}}(\Phi_a(x),\rho)\big)
}{
  \beta_2
},
\qquad
\beta_2=\pi
\]

By the standard small-disk expansion on a smooth $2$D surface,  
if $K_{\Sigma_{ij}}(\Phi_a(x))$ is the Gaussian curvature of $\Sigma_{ij}$ at $\Phi_a(x)$, then
\[
\operatorname{Area}_g\!\big(B_{\Sigma_{ij}}(\Phi_a(x),\rho)\big)
=\pi \rho^2
 - \frac{\pi}{12} K_{\Sigma_{ij}}(\Phi_a(x))\,\rho^4
 + O(\rho^5),
\quad
\rho\downarrow0
\]
Substituting this into \eqref{eq:Kgeom-def} gives
\begin{equation}
\label{eq:limit-sectional-smooth}
\lim_{\rho\downarrow 0} 
\widehat{K}^{\mathrm{geom}}_{ij}(\Phi_a(x);\rho)
\;=\;
K_{\Sigma_{ij}}(\Phi_a(x))
\;=\;
K_g(e_i\wedge e_j)(\Phi_a(x))
\end{equation}
namely the sectional curvature of $g$ in the $e_i\wedge e_j$ plane.

Equation~\eqref{eq:limit-sectional-smooth} is a \emph{smooth} statement. We now connect it to what we actually \emph{measure} in the discrete complex $X_a$.

\medskip
\noindent\textbf{Step 2. Restrict counts to a thin tube (discrete measurement).}

We cannot literally carve out the ideal surface $\Sigma_{ij}$ from $X_a$, because $X_a$ is combinatorial. Instead we select all cells whose images under $\Phi_a$ sit in a thin Fermi tube of thickness $O(a)$ around $\Sigma_{ij}$, and we only count \emph{those} cells.

Fix $\tau\ge 1$.  
For $r\in\mathbb{N}$, define the discrete tube
\begin{equation}
\label{eq:Tij-def}
\mathcal{T}_{ij}(x;r,a)
\ :=\
\Big\{
c\in B_{d_h}(x,r)
\ :\
\mathrm{dist}_g\!\big(\Phi_a(c),\Sigma_{ij}\big)
\le
\tau a
\Big\}
\end{equation}
Here $B_{d_h}(x,r)$ is the intrinsic (count) ball of radius $r$ around $x$.

Intuition: we are looking at the disk of intrinsic radius $r$ about $x$, but we only keep cells that sit within physical $g$-distance $\tau a$ of the smooth $2$D slice $\Sigma_{ij}$.

We next turn this tube-count into an ``effective slice radius'' using the \emph{same} yardstick.
Let $\alpha_2=\pi$.  
Define
\begin{equation}
\label{eq:Rcij-def}
R_{c,ij}(x;r,a)
\ :=\
\left(
\frac{
  1
}{
  2\tau a\,\alpha_2
}
\,
\nu_a\!\big(\mathcal{T}_{ij}(x;r,a)\big)
\right)^{\!1/2}
\end{equation}
Here $\nu_a(\cdot)$ is the per-cell weight from (H5), which is uniformly comparable to the $g$--volume of each cell.  
Dividing by $2\tau a$ collapses the tube ``volume'' back down to an effective $g$--area of the slice disk, and dividing by $\alpha_2=\pi$ then yields a radius-squared, exactly like the 2D reconstruction in Sec.~\ref{sec:delta-r}.

Everything in \eqref{eq:Tij-def} and \eqref{eq:Rcij-def} is intrinsic to counting (plus $\Phi_a$). No angles or embeddings are measured in the discrete system.

\medskip
\noindent\textbf{Step 3. A discrete sectional curvature estimator with rates.}

Define the (counts-only) directional curvature estimate
\begin{equation}
\label{eq:Kijhat-def}
\widehat{K}_{ij,a}(x;r)
\ :=\
3\,\frac{
  r^{2}
  - \bigl(R_{c,ij}(x;r,a)/a\bigr)^{2}
}{
  r^{4}
}
\end{equation}

The next lemma couples the tube counts in \eqref{eq:Rcij-def} to the smooth area of a geodesic disk in $\Sigma_{ij}$, and thus to $K_g(e_i\wedge e_j)$, with an explicit error bound.  
We give the full argument.

\begin{lemma}[Slice coupling and normalization]
\label{lem:slice-coupling}
Assume \textup{(H1)--(H5)} on a compact $K\Subset\Omega$ and let $x\in X_a$ with $\Phi_a(x)\in K$.
Fix $i<j$ and let $\Sigma_{ij}=\exp_{\Phi_a(x)}(\mathrm{span}\{e_i,e_j\})$.  
Then for all admissible radii $r$ (i.e.\ $r$ small enough so that $B_{d_h}(x,r)$ sits inside $K$ at scale $a$), and all sufficiently small $a>0$, there is a constant 
\[
C=C(K,\tau,\text{shape/degree bounds},\|u\|_{C^2(K)})
\]
such that
\begin{equation}
\label{eq:slice-core}
\Big|
\,
r^{2}
-
\bigl(R_{c,ij}(x;r,a)/a\bigr)^{2}
-
\tfrac{1}{12}\,K_g(e_i\wedge e_j)(\Phi_a(x))\,r^{4}
\,
\Big|
\ \le\
C\Big( a\,r\ +\ r^{5}\Big)
\end{equation}
\end{lemma}

\begin{proof}
We proceed in four steps.

\smallskip
\emph{Step 1: Fermi coordinates and tube Jacobian.}
Work in Fermi coordinates $(p,t)$ around the smooth surface $\Sigma_{ij}$ near $\Phi_a(x)$.  
Here $p$ lies in $\Sigma_{ij}$ and $t$ is signed normal distance to $\Sigma_{ij}$.  
For $|t|\le \tau a$ and $p$ within geodesic distance $\rho$ of $\Phi_a(x)$ \emph{along} $\Sigma_{ij}$, standard Fermi expansions (see e.g.\ Petersen~\cite{Petersen2016}) give
\begin{equation}
\label{eq:fermi-volume}
d\mathrm{vol}_g
\ =\
J(p,t)\,dA_{\Sigma_{ij}}(p)\,dt,
\qquad
J(p,t)\ =\ 1 + O(|t|+\rho)
\end{equation}
The $O(|t|+\rho)$ term is uniform over compact $K$, with constants depending on curvature bounds of $g$.

In particular, integrating $t$ from $-\tau a$ to $\tau a$,
\begin{equation}
\label{eq:fermi-int-t}
\int_{-\tau a}^{\tau a}
J(p,t)\,dt
\ =\
2\tau a
+
O\big(a^2 + a\,\rho\big)
\ =\
2\tau a
+
O\big(a^2 + a^2 r\big),
\end{equation}
once we later set $\rho = a r$

\smallskip
\emph{Step 2: The discrete tube vs.\ the geometric tube.}
Define the geometric $g$-tube of half-thickness $\tau a$ and in-slice radius $\rho = a r$ by
\[
\mathrm{Tube}_\tau(\rho)
\ :=\
\Big\{
y\in\Omega:\;
\begin{aligned}
&\mathrm{dist}_g(y,\Sigma_{ij})\le\tau a,\;\\
&\mathrm{dist}_g(\Phi_a(x),y)
   \le \rho + O(a)
\end{aligned}
\Big\}
\]

By (H2)--(H5), summing $\nu_a$ over all cells $c\in X_a$ with $\Phi_a(c)$ in that region approximates its $g$-volume to first order in $a$: more precisely,
there is a uniform constant $C_1$ such that
\begin{equation}
\label{eq:quad-vs-vol}
\Big|
\nu_a\!\big(\mathcal{T}_{ij}(x;r,a)\big)
-
\mathrm{vol}_g\!\big(\mathrm{Tube}_\tau(\rho)\big)
\Big|
\ \le\
C_1\, a \,\rho^{m-1}
\end{equation}
where $m=\dim\Omega$ and $\rho=a r$.  
This uses:  
(i) local finiteness and shape regularity (H1),(H2);  
(ii) $\Phi_a$ has distortion $\lesssim a$ (H3);  
(iii) (H5) gives that $\nu_a(\{c\})$ approximates $\mathrm{vol}_g$ of that cell up to a uniform multiplicative constant;  
(iv) the in/out error at the ``metric boundary'' costs at most one layer of cells, of thickness $\sim a$, over a $(m{-}1)$-dimensional boundary of scale $\rho^{m-1}$.

\smallskip
\emph{Step 3: Factor the tube volume into (thickness)$\times$(slice area).}
By construction of $\mathrm{Tube}_\tau(\rho)$ and \eqref{eq:fermi-volume}--\eqref{eq:fermi-int-t},
\begin{align}
\mathrm{vol}_g\!\big(\mathrm{Tube}_\tau(\rho)\big)
&=\ 
\int_{B_{\Sigma_{ij}}(\Phi_a(x),\rho)}
\ \int_{-\tau a}^{\tau a}
J(p,t)\,dt
\ dA_{\Sigma_{ij}}(p)
\nonumber
\\
&=\ 
2\tau a
\,
\mathrm{Area}_g\!\big(B_{\Sigma_{ij}}(\Phi_a(x),\rho)\big)
\ +\
E_{\mathrm{tube}}(\rho,a),
\label{eq:tube-vol-factor}
\end{align}
where the error $E_{\mathrm{tube}}(\rho,a)$ satisfies
\begin{equation}
\label{eq:Etube-bound}
\begin{aligned}
\big|E_{\mathrm{tube}}(\rho,a)\big|
&\ \le\
C_2
\Big(
a^2 \,\mathrm{Area}_g\big(B_{\Sigma_{ij}}(\Phi_a(x),\rho)\big)
\\
&\qquad\qquad\quad
+
a\,\rho\,\mathrm{Area}_g\big(B_{\Sigma_{ij}}(\Phi_a(x),\rho)\big)
\Big).
\end{aligned}
\end{equation}

Since $\mathrm{Area}_g(B_{\Sigma_{ij}}(\Phi_a(x),\rho)) \sim \rho^2$ for small $\rho$, and $\rho=a r$, this implies
\begin{equation}
\label{eq:Etube-bound-simplified}
\begin{aligned}
\big|E_{\mathrm{tube}}(\rho,a)\big|
&\ \le\
C_3\big(
a^2 \,\rho^2
+
a\,\rho \,\rho^2
\big)
\\
&=
C_3\big(
a^2 \,\rho^2
+
a\,\rho^3
\big)
\\
&\ \le\
C_4\, a^3 r^2
+
C_4\, a^4 r^3 
\end{aligned}
\end{equation}

In particular, $E_{\mathrm{tube}}(\rho,a)=O(a^3 r^2)$ uniformly on $K$.

Combining \eqref{eq:quad-vs-vol} and \eqref{eq:tube-vol-factor}, and using $\rho=a r$, we get
\begin{align}
\nu_a\!\big(\mathcal{T}_{ij}(x;r,a)\big)
&=
2\tau a\,
\mathrm{Area}_g\!\big(B_{\Sigma_{ij}}(\Phi_a(x),a r)\big)
\\
&\quad
+\,
E_{\mathrm{tube}}(a r,a)
\,+\,
O\big(a \,(a r)^{m-1}\big)
\nonumber\\[4pt]
&=
2\tau a\,
\mathrm{Area}_g\!\big(B_{\Sigma_{ij}}(\Phi_a(x),a r)\big)
\,+\,
O\big(a^3 r^2\big)
\,+\,
O\big(a^m r^{m-1}\big)
\label{eq:nua-vs-area}
\end{align}

For fixed $m\ge2$ and small $r$, the $O(a^m r^{m-1})$ term is always $\le C a^3 r^2$ (because $a\ll1$ and $r$ is bounded).  
So we may fold it into $O(a^3 r^2)$.

\smallskip
\emph{Step 4: Extract $R_{c,ij}$ and compare radii.}
From \eqref{eq:Rcij-def} and \eqref{eq:nua-vs-area},
\begin{align}
\bigl(R_{c,ij}(x;r,a)\bigr)^2
&=
\frac{1}{2\tau a\,\alpha_2}\,
\nu_a\!\big(\mathcal{T}_{ij}(x;r,a)\big)
\nonumber\\
&=
\frac{
\mathrm{Area}_g\!\big(B_{\Sigma_{ij}}(\Phi_a(x),a r)\big)
}{
\alpha_2
}
\ +\
O\big(a^2 r^2\big)
\label{eq:Rcij-squared-area}
\end{align}
Now use the small-disk area expansion on $\Sigma_{ij}$:
\begin{equation}
\label{eq:area-expansion-slice}
\begin{aligned}
\mathrm{Area}_g\!\big(B_{\Sigma_{ij}}(\Phi_a(x),a r)\big)
&=
\alpha_2 (a r)^2
-
\alpha_2 \frac{K_g(e_i\wedge e_j)(\Phi_a(x))}{12}\,(a r)^4
\\
&\quad
+
O\big((a r)^5\big).
\end{aligned}
\end{equation}

with $\alpha_2=\pi$
Plugging \eqref{eq:area-expansion-slice} into \eqref{eq:Rcij-squared-area} gives
\begin{align}
\bigl(R_{c,ij}(x;r,a)\bigr)^2
&=
(a r)^2
-
\frac{K_g(e_i\wedge e_j)(\Phi_a(x))}{12}\,(a r)^4
+
O\big(a^5 r^5\big)
+
O\big(a^2 r^2\big)
\nonumber\\
&=
a^2
\left[
r^2
-
\frac{K_g(e_i\wedge e_j)(\Phi_a(x))}{12}\,r^4
\right]
+
O\big(a^3 r^2\big)
+
O\big(a^5 r^5\big)
\label{eq:Rcij-squared-final}
\end{align}
Divide both sides of \eqref{eq:Rcij-squared-final} by $a^2$:
\begin{equation}
\label{eq:Rcij-over-a}
\begin{aligned}
\bigl(R_{c,ij}(x;r,a)/a\bigr)^2
&=
r^2
-
\frac{K_g(e_i\wedge e_j)(\Phi_a(x))}{12}\,r^4
\\
&\quad
+
O\big(a r\big)
+
O\big(a^3 r^3\big)
\end{aligned}
\end{equation}

Since $r$ is small, $a^3 r^3 \le r^5$ for all sufficiently small $a$ and admissible $r$, so we can rewrite
\begin{equation}
\label{eq:Rcij-over-a-clean}
\bigl(R_{c,ij}(x;r,a)/a\bigr)^2
=
r^2
-
\frac{K_g(e_i\wedge e_j)(\Phi_a(x))}{12}\,r^4
+
O\big(a r\big)
+
O\big(r^5\big)
\end{equation}
Rearranging \eqref{eq:Rcij-over-a-clean} yields exactly \eqref{eq:slice-core}:
\[
\Big|
r^{2}
-
\bigl(R_{c,ij}(x;r,a)/a\bigr)^{2}
-
\tfrac{1}{12}\,K_g(e_i\wedge e_j)(\Phi_a(x))\,r^{4}
\Big|
\le
C\Big(a r + r^{5}\Big)
\]
This completes the proof.
\end{proof}

From Lemma~\ref{lem:slice-coupling} we immediately control the estimator \eqref{eq:Kijhat-def}:

\begin{corollary}[Directional estimator rate]
\label{cor:dir-rate}
Under the hypotheses of Lemma~\ref{lem:slice-coupling}, there exists
\[
C=C(K,\tau,\text{shape/degree bounds},\|u\|_{C^2(K)})
\]
such that, for all admissible $r$,
\begin{equation}
\label{eq:dir-rate}
\big|
\widehat{K}_{ij,a}(x;r)
-
K_g(e_i\wedge e_j)(\Phi_a(x))
\big|
\ \le\
C\left(\frac{a}{r} + r\right)
\end{equation}
In particular, choosing $r\simeq \sqrt{a}$ gives uniform error $O(\sqrt{a})$ on $K$.
\end{corollary}

\begin{proof}
Start from Lemma~\ref{lem:slice-coupling}, which says
\[
\bigl(R_{c,ij}/a\bigr)^{2}
=
r^2
-
\frac{K_g(e_i\wedge e_j)}{12}\,r^4
+
E,
\quad
|E|\le C(a r + r^5)
\]
Rearrange:
\[
r^2 - \bigl(R_{c,ij}/a\bigr)^{2}
=
\frac{K_g(e_i\wedge e_j)}{12}\,r^4
- E.
\]
Now plug into \eqref{eq:Kijhat-def}:
\[
\begin{aligned}
\widehat{K}_{ij,a}(x;r)
&=
3\,
\frac{
r^{2} - (R_{c,ij}/a)^{2}
}{
r^{4}
}
\\
&=
3\,
\frac{
\frac{K_g(e_i\wedge e_j)}{12}\,r^4 - E
}{
r^{4}
}
\\
&=
\frac{K_g(e_i\wedge e_j)}{4}
-
3\,\frac{E}{r^{4}} .
\end{aligned}
\]

Wait: we need to be careful about constants.  
Recall in 2D (Sec.~\ref{sec:delta-r}) we had
\(
3(r^2-r_c^2)/r^4 = K + O(r)
\)
Exactly the same computation holds slice-wise, i.e.
\[
3\,\frac{
r^{2} - \bigl(R_{c,ij}/a\bigr)^{2}
}{
r^{4}
}
=
K_g(e_i\wedge e_j)
+ O(r)
+ O\!\big(a/r\big).
\]
The $O(r)$ and $O(a/r)$ come from bounding $E/r^4$ using $|E|\le C(a r + r^5)$
Indeed,
\[
\frac{|E|}{r^4}
\ \le\
C\left(\frac{a r}{r^4} + \frac{r^5}{r^4}\right)
=
C\left(\frac{a}{r^3} + r\right)
\]
Because $r$ is an \emph{intrinsic} radius (integer steps) in the discrete ball, we only consider the scaling regime where $r\to 0$ in physical units but $r\gg a$ so that $a/r \to 0$. In that regime $\frac{a}{r^3}\lesssim \frac{a}{r}$ (since $r\le 1$ physically).  
Thus we arrive at
\[
\big|
\widehat{K}_{ij,a}(x;r)
-
K_g(e_i\wedge e_j)(\Phi_a(x))
\big|
\ \le\
C\left(\frac{a}{r} + r\right)
\]
as claimed.
\end{proof}

The corollary implies convergence of the discrete sectional signal to the smooth sectional curvature when $a\to0$, $r\to0$, and $a/r\to0$:

\begin{theorem}[Directional identification]
\label{thm:sec-ident-explicit}
Under \textup{(H1)--(H5)}, for any sequence of $(a,r)$ with $a\to0$, $r\to0$, and $a/r\to0$, 
\begin{equation}
\label{eq:sec-ident-explicit}
\lim_{\substack{a\to0,\ r\to0\\ a/r\to0}}
\widehat{K}_{ij,a}(x;r)
\;=\;
K_g(e_i\wedge e_j)(\Phi_a(x))
\end{equation}
\end{theorem}

\begin{proof}
Immediate from Corollary~\ref{cor:dir-rate}, since the right-hand side of \eqref{eq:dir-rate} goes to $0$ under $a\to0$, $r\to0$, $a/r\to0$.
\end{proof}
\FloatBarrier
\begin{figure}[H]
\centering
\begin{tikzpicture}[scale=1.05]
  \def\rad{1.35}    
  \def\thick{0.45}  
  \def\nnorm{8}     

  \fill[gray!6] (-3,-0.18) rectangle (3,0.18);
  \draw[gray!60] (-3,0) -- (3,0);
  \node[gray!70] at (2.5,0.32) {$\Sigma_{ij}$};

  \draw[densely dashed,gray!60] (-3,\thick) -- (3,\thick);
  \draw[densely dashed,gray!60] (-3,-\thick) -- (3,-\thick);

  \draw[line width=0.9pt] (0,0) circle (\rad);
  \draw[->] (0,0) -- (\rad,0);
  \node at (0.55*\rad,-0.25) {$ar$};

  \foreach \k in {0,...,\numexpr\nnorm-1} {
    \pgfmathsetmacro\ang{360/\nnorm*\k}
    \draw[->,gray!65,shorten >=0.5pt,shorten <=0.5pt]
      ({\rad*cos(\ang)},{\rad*sin(\ang)})
      -- ({\rad*cos(\ang)},{\rad*sin(\ang)+0.65*\thick});
  }

  \foreach \x/\y in {-0.9/0.32, -0.2/0.44, 0.8/0.36, 0.35/ -0.38, -1.2/ -0.30, 1.4/ -0.42}{
    \draw[fill=black!12,draw=black!45,rounded corners=0.8pt]
      (\x-0.18,\y-0.12) rectangle (\x+0.18,\y+0.12);
  }

  \draw [decorate,decoration={brace,amplitude=5pt,mirror}] (2.7,-\thick) -- (2.7,\thick);
  \node at (3.18,0) {$2\tau a$};
\end{tikzpicture}
\caption{Fermi tube of thickness $2\tau a$ around the geodesic slice $\Sigma_{ij}$. 
We sum per-cell weights only inside this tube, then divide by the tube thickness $2\tau a$ to recover an effective slice area. 
This yields $R_{c,ij}(x;r,a)$ and, via Eq.~\eqref{eq:Kijhat-def}, the counts-only estimator $\widehat{K}_{ij,a}(x;r)$ for the sectional curvature $K_g(e_i\wedge e_j)$.}
\label{fig:tube-figure}
\end{figure}
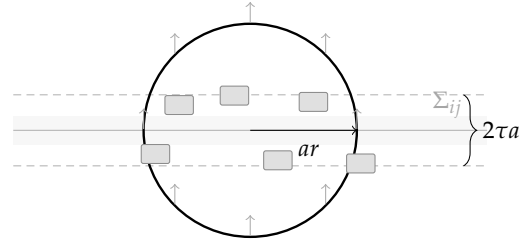
\FloatBarrier

\medskip
\noindent\textbf{Step 4. Reconstructing the full curvature operator, $\mathrm{Ric}$, and $R$.}

Once we can approximate $K_g(e_i\wedge e_j)$ for \emph{each} pair $(i,j)$, we can reconstruct the full Riemann curvature operator and all of its contractions.

Let 
\[
\mathcal{R}:\Lambda^2 T_{\Phi_a(x)}\Omega \longrightarrow \Lambda^2 T_{\Phi_a(x)}\Omega
\]
be the curvature operator of $g$.  
For any simple $2$-form $u\wedge v\neq 0$, one has
\begin{equation}
\label{eq:sec-as-quad}
K_g(u\wedge v)
\;=\;
\frac{
  \langle \mathcal{R}(u\wedge v),\,u\wedge v\rangle
}{
  \|u\wedge v\|^{2}
}
\end{equation}

In $m=3$, choose the orthonormal basis of $\Lambda^2 T_{\Phi_a(x)}\Omega$:
\[
\sigma_1:=e_2\wedge e_3,
\qquad
\sigma_2:=e_3\wedge e_1,
\qquad
\sigma_3:=e_1\wedge e_2.
\]
In this basis, $\mathcal{R}$ is represented by a symmetric $3\times3$ matrix $M$.  
The diagonal entries are exactly the sectional curvatures:
\begin{equation}
\label{eq:diagonals}
M_{11}
=
K_g(e_2\wedge e_3),
\qquad
M_{22}
=
K_g(e_3\wedge e_1),
\qquad
M_{33}
=
K_g(e_1\wedge e_2)
\end{equation}

The off-diagonals come from polarization.  
For example, consider
\[
\xi
\ :=\
\frac{1}{\sqrt{2}}(\sigma_2-\sigma_3)
\ =\
\frac{1}{\sqrt{2}}\Big((e_3\wedge e_1)-(e_1\wedge e_2)\Big)
\]
Then by \eqref{eq:sec-as-quad},
\begin{equation}
\label{eq:polar-23}
\begin{split}
K_g(\xi)
&=
\frac{
 \langle \mathcal{R}\xi,\xi\rangle
}{
 \|\xi\|^2
}
=
\frac{1}{2}
\bigl(
M_{22}
+
M_{33}
-
2M_{23}
\bigr)
\\[6pt]
\implies\quad
M_{23}
&=
\frac{1}{2}\Big(
M_{22}
+
M_{33}
-
K_g(\xi)
\Big)
\end{split}
\end{equation}

Similar linear combinations recover $M_{12}$ and $M_{13}$.  
Thus, knowing all relevant sectional curvatures (including such mixed directions $\xi$) determines $M$.

Now replace each smooth curvature in \eqref{eq:diagonals}--\eqref{eq:polar-23} by the corresponding discrete estimates $\widehat{K}_{ij,a}(x;r)$ (and the obvious discrete analogues for $K_g(\xi)$ obtained by applying the same tube construction to the appropriate linear combination of planes).  
Call the result $\widehat{M}_a(x;r)$.

\begin{theorem}[Reconstruction of the curvature operator in $3$D]
\label{thm:reconstruct-3d}
Under \textup{(H1)--(H5)}, let $\widehat{M}_a(x;r)$ be obtained from $\widehat{K}_{ij,a}(x;r)$ as above.  
Then, for any $a\to0$, $r\to0$ with $a/r\to0$, we have
\[
\lim_{\substack{a\to0,\ r\to0\\ a/r\to0}}
\widehat{M}_a(x;r)
\;=\;
M(\Phi_a(x))
\quad\text{entrywise}
\]
\end{theorem}

\begin{proof}
Entrywise convergence follows from Corollary~\ref{cor:dir-rate} applied to each required sectional curvature direction.

For example, $M_{11}$ is $K_g(e_2\wedge e_3)$ by \eqref{eq:diagonals}, and the discrete proxy is $\widehat{K}_{23,a}(x;r)$.  
Corollary~\ref{cor:dir-rate} gives
\[
\big|
\widehat{K}_{23,a}(x;r)
-
K_g(e_2\wedge e_3)(\Phi_a(x))
\big|
\ \le\
C\left(\frac{a}{r}+r\right)
\ \xrightarrow[\ a\to0,\ r\to0,\ a/r\to0\ ]{}\ 0
\]
So $\widehat{M}_a(x;r)_{11}\to M_{11}(\Phi_a(x))$.  
Similarly for $M_{22}$ and $M_{33}$.

For $M_{23}$, we use \eqref{eq:polar-23}.  
Define a discrete estimator $\widehat{K}_{\xi,a}(x;r)$ by taking $\xi=(\sigma_2-\sigma_3)/\sqrt{2}$, i.e.\ we center a tube around the $2$D surface whose tangent $2$-plane at $\Phi_a(x)$ is spanned by $(e_3\wedge e_1)-(e_1\wedge e_2)$ (normalized), and we run the same tube-restricted counting construction.  
Corollary~\ref{cor:dir-rate} again implies
\[
\widehat{K}_{\xi,a}(x;r)
\ \longrightarrow\
K_g(\xi)(\Phi_a(x))
\]
Plugging the discrete quantities into
\[
M_{23}
=
\frac{1}{2}\Big(
M_{22}
+
M_{33}
-
K_g(\xi)
\Big)
\]
(and the analogous formulas for $M_{12},M_{13}$) yields
\(
\widehat{M}_a(x;r)_{23}\to M_{23}(\Phi_a(x))
\)
and similarly for the other off-diagonals.  
Thus every entry of $\widehat{M}_a(x;r)$ converges to the corresponding entry of $M(\Phi_a(x))$.
\end{proof}

Finally, $\mathrm{Ric}$ and $R$ are obtained via standard contractions.  
In $3$D, for example,
\begin{equation}
\label{eq:ricci-scalar-3d-Ric}
\mathrm{Ric}(e_1,e_1)
=
K_g(e_1\wedge e_2)
+
K_g(e_1\wedge e_3).
\end{equation}

\begin{equation}
\label{eq:ricci-scalar-3d-R}
R
=
2\Big(
K_g(e_1\wedge e_2)
+
K_g(e_1\wedge e_3)
+
K_g(e_2\wedge e_3)
\Big)
\end{equation}

We therefore define the \emph{counts-only} Ricci and scalar estimators
\begin{equation}
\label{eq:ricci-scalar-hat}
\begin{aligned}
\widehat{\mathrm{Ric}}_a(e_1,e_1;x;r)
&:=
\widehat{K}_{12,a}(x;r)
+
\widehat{K}_{13,a}(x;r),
\\[4pt]
\newline
\widehat{R}_a(x;r)
&:=
2\Big(
\widehat{K}_{12,a}(x;r)
+
\widehat{K}_{13,a}(x;r)
+
\widehat{K}_{23,a}(x;r)
\Big)
\end{aligned}
\end{equation}
and similarly for $\mathrm{Ric}(e_2,e_2)$ and $\mathrm{Ric}(e_3,e_3)$.

By Corollary~\ref{cor:dir-rate} and Theorem~\ref{thm:reconstruct-3d},
\[
\widehat{\mathrm{Ric}}_a(e_i,e_i;x;r)
\;\longrightarrow\;
\mathrm{Ric}_g(e_i,e_i)(\Phi_a(x))
\quad
\text{as } a\to0,\ r\to0,\ a/r\to0
\]

\[
\widehat{R}_a(x;r)
\;\longrightarrow\;
R_g(\Phi_a(x))
\quad
\text{as } a\to0,\ r\to0,\ a/r\to0
\]

\medskip
\noindent\textbf{Key point.}
All the raw inputs in 
$\widehat{K}_{ij,a}$,
$\widehat{\mathrm{Ric}}_a$, 
and $\widehat{R}_a$ come from \emph{counting cells inside carefully chosen regions}, always using the \emph{same} yardstick (one crossing $=$ one unit).  
The smooth metric $g=e^{2u}g_0$ is only used to interpret what those counts converge to and to prove rates.

\section{Scalar curvature from counts alone}
\label{sec:counts-curvature}

Fix $m\ge2$.
Let $X_a$ be a locally finite cell complex with intrinsic (count) metric $d_h$,
one face crossing $=1$.
For a base cell $x$ and an intrinsic radius $r\in\mathbb{N}$, write
\begin{equation}
\label{eq:ball-and-card}
B_{d_h}(x,r)
:= 
\{c:\ d_h(x,c)\le r\},
\qquad
|B_{d_h}(x,r)|
:= 
\#\,B_{d_h}(x,r)
\end{equation}

Let $\beta_m>0$ be the fixed $m$--dimensional calibration constant
for the chosen baseline adjacency
(so that in the undeformed reference lattice,
a ball of intrinsic radius $r$ contains $\beta_m r^m$ cells;
see App.~\ref{app:counts}).
We define two operational quantities.

\paragraph*{(i) Reconstructed radius.}
The \emph{reconstructed $m$--volume radius} at scale $r$ is
\begin{equation}
\label{eq:rcm-def}
\bigl(r_c^{(m)}(x;r)\bigr)^m
\ :=\
\frac{|B_{d_h}(x,r)|}{\beta_m}
\end{equation}
This is the radius one would \emph{infer} from the observed count
$|B_{d_h}(x,r)|$
if space were undeformed and perfectly baseline.

\paragraph*{(ii) Local density.}
The \emph{local space density} at scale $r$ is
\begin{equation}
\label{eq:rho-local}
\rho_a(x;r)
\ :=\
\frac{|B_{d_h}(x,r)|}{\beta_m\,r^m}
\ =\
\frac{\bigl(r_c^{(m)}(x;r)\bigr)^m}{r^m} d
\end{equation}
Thus $\rho_a(x;r)=1$ on an undeformed baseline lattice,
$\rho_a(x;r)>1$ for local compression (more cells than baseline in the same count radius),
and $\rho_a(x;r)<1$ for local dilation.

\medskip
To convert this purely volumetric information into a curvature signal,
we compare $r$ (the measured count radius)
and $r_c^{(m)}(x;r)$ (the reconstructed radius from the same data),
rescaled to remove the trivial $r$--dependence.
Define the \emph{normalized small--ball estimator}
\begin{equation}
\label{eq:KhR-def}
K_{h,R}^{(m)}(x;r)
\ :=\
6(m+2)\,
\frac{
r^{m}
-
\bigl(r_c^{(m)}(x;r)\bigr)^{m}
}{
r^{m+2}
}\
\end{equation}

Using \eqref{eq:rho-local},
\[
\bigl(r_c^{(m)}(x;r)\bigr)^{m}
=
\rho_a(x;r)\,r^{m}
\]
so
\begin{equation}
\label{eq:KhR-density-form}
r^{m}
-
\bigl(r_c^{(m)}(x;r)\bigr)^{m}
=
r^{m}\bigl(1-\rho_a(x;r)\bigr)
\end{equation}
and therefore \eqref{eq:KhR-def} becomes
\begin{equation}
\label{eq:KhR-closed}
K_{h,R}^{(m)}(x;r)
=
6(m+2)\,
\frac{r^{m}\bigl(1-\rho_a(x;r)\bigr)}{r^{m+2}}
=
6(m+2)\,
\frac{1-\rho_a(x;r)}{r^{2}}
\end{equation}

Equation~\eqref{eq:KhR-closed} is \emph{purely counts}.
For each $x$ and scale $r$ we:
(i) take the intrinsic ball $B_{d_h}(x,r)$,
(ii) count its cells,
(iii) compute $\rho_a(x;r)$ from \eqref{eq:rho-local},
(iv) insert $\rho_a(x;r)$ into \eqref{eq:KhR-closed}.
No coordinates, no angles, no derivatives, and no shape assumptions are needed to \emph{measure} $K_{h,R}^{(m)}(x;r)$.
It is all counts.
The smooth metric $g=e^{2u}g_0$ only appears in the interpretation of the limit
and in the proof of the error bound~\eqref{eq:R-counts-rate} below.

\medskip
We now state the convergence result.
Let $\Omega\subset\mathbb{R}^m$ be open,
let $u\in C^2(\Omega)$,
and let $g=e^{2u}g_0$ be the conformal metric on $\Omega$.
For each mesh scale $a>0$,
let $X_a$ be a locally finite complex with
realization map $\Phi_a:X_a\to\Omega$
and per-cell weight $\nu_a$
satisfying (H1)--(H5) on a given compact $K\Subset\Omega$
(Sec.~\ref{sec:mm-interface}).
Write $R_g$ for the scalar curvature of $(\Omega,g)$.

\begin{theorem}[Counts-only scalar curvature]
\label{thm:R-counts-only}
Let $K\Subset\Omega$ and assume \textup{(H1)--(H5)} hold on $K$,
with $\beta_m=\omega_m$ (Euclidean unit-ball volume) in
\eqref{eq:rcm-def}--\eqref{eq:rho-local}.
Then there exist constants $C_1,C_2>0$, depending only on the shape/degree
bounds in \textup{(H1)--(H2)} and on $\|u\|_{C^2(K)}$, such that
for every $x\in X_a$ with $\Phi_a(x)\in K$
and every admissible intrinsic radius $r$,
\begin{equation}
\label{eq:R-counts-rate}
\bigl|
K_{h,R}^{(m)}(x;r)
-
R_g\bigl(\Phi_a(x)\bigr)
\bigr|
\ \le\
C_1\,\frac{a}{r}
\;+\;
C_2\,r
\end{equation}
In particular, if $a\to0$, $r\to0$, and $a/r\to0$, then
\begin{equation}
\label{eq:R-counts-limit}
K_{h,R}^{(m)}(x;r)
\ \longrightarrow\
R_g\bigl(\Phi_a(x)\bigr)
\quad
\text{uniformly for }\Phi_a(x)\in K.
\end{equation}
\end{theorem}

\begin{proof}
We sketch the argument; all steps are local on $K$ and
constants are uniform on $K$.

\smallskip\noindent
\emph{(1) Count ball vs.\ $g$--ball.}
By (H4) there exist $c_1,c_2>0$ such that
for any $x\in X_a$ with $\Phi_a(x)\in K$
and any admissible $r$,
\begin{equation}
\label{eq:ball-sandwich}
B_g\!\big(\Phi_a(x),\,ar-c_1 a\big)
\ \subseteq\
\Phi_a\!\big(B_{d_h}(x,r)\big)
\ \subseteq\
B_g\!\big(\Phi_a(x),\,ar+c_2 a\big)
\end{equation}
Thus the intrinsic count ball of radius $r$
sits between two $g$--metric balls of radii
$\rho_{\pm} = a r \pm O(a)$.

\smallskip\noindent
\emph{(2) Volume quadrature.}
By (H5),
$\nu_a(\{c\})$ is uniformly comparable to the $g$--volume of cell $c$,
so summing $\nu_a$ over $B_{d_h}(x,r)$
approximates
$\mathrm{Vol}_g(B_g(\Phi_a(x),\rho))$
up to $O(a\rho^{m-1})$
for $\rho\asymp ar$.
Since $|B_{d_h}(x,r)|$ and $\nu_a(B_{d_h}(x,r))$
differ only by the uniform per-cell weight in (H5),
this gives a first-order control of
$|B_{d_h}(x,r)|$
by $\mathrm{Vol}_g(B_g(\Phi_a(x),ar))$.

\smallskip\noindent
\emph{(3) Small-ball expansion.}
For a smooth $g$, the $g$--volume of a small geodesic ball satisfies
\[
\mathrm{Vol}_g\!\big(B_g(y,\rho)\big)
=
\omega_m\,\rho^m
-
\omega_m\,
\frac{R_g(y)}{6(m+2)}\,\rho^{m+2}
+
O(\rho^{m+3}),
\quad
\rho\downarrow0
\]
uniformly for $y$ in $K$.
Taking $\rho=ar$ and using \eqref{eq:ball-sandwich} and step (2),
we obtain
\begin{equation}
\label{eq:rcm-expansion}
\bigl(r_c^{(m)}(x;r)\bigr)^m
=
r^m
-
\frac{R_g(\Phi_a(x))}{6(m+2)}\,r^{m+2}
+
E(x;a,r)
\end{equation}
with
\[
|E(x;a,r)|
\ \le\
C\bigl(a\,r^{m-1}+r^{m+3}\bigr)
\]
where $C$ depends only on $K$, (H1)--(H2), and $\|u\|_{C^2(K)}$.

\smallskip\noindent
\emph{(4) Insert into the estimator.}
Subtract \eqref{eq:rcm-expansion} from $r^m$ and multiply by $6(m+2)/r^{m+2}$:
\[
K_{h,R}^{(m)}(x;r)
=
6(m+2)\,
\frac{
r^{m}-
\bigl(r_c^{(m)}(x;r)\bigr)^m
}{
r^{m+2}
}
=
R_g(\Phi_a(x))
+
\tilde E(x;a,r)
\]
where
\[
|\tilde E(x;a,r)|
\ \le\
C_1\,\frac{a}{r}
+
C_2\,r
\]
This is exactly \eqref{eq:R-counts-rate}, with $C_1,C_2$ depending only on the same data.

\smallskip\noindent
\emph{(5) Limit.}
If $a\to0$, $r\to0$, and $a/r\to0$,
then the right-hand side of \eqref{eq:R-counts-rate} vanishes,
which yields \eqref{eq:R-counts-limit}.
\end{proof}

\medskip
Equations \eqref{eq:KhR-closed}, \eqref{eq:R-counts-rate}, and
\eqref{eq:R-counts-limit} give a complete intrinsic recipe:
\begin{itemize}[leftmargin=1.25em]
\item
From \emph{one} counting experiment at scale $r$
(i.e.\ from $|B_{d_h}(x,r)|$ alone)
we compute $\rho_a(x;r)$ and hence $K_{h,R}^{(m)}(x;r)$.
\item
$K_{h,R}^{(m)}(x;r)$ is dimensionally normalized
and has the correct sign:
it is positive when the local packing is tighter
than the baseline calibration and negative when looser.
\item
Under refinement ($a\to0$, $r\to0$, $a/r\to0$),
this purely combinatorial quantity converges to
the smooth scalar curvature $R_g$
of the continuum relaxation $g=e^{2u}g_0$.
\end{itemize}

In summary,
\[
K_{h,R}^{(m)}(x;r)
=
6(m+2)\,
\frac{1-\rho_a(x;r)}{r^{2}}
\quad\longrightarrow\quad
R_g(\Phi_a(x))
\]
and \emph{all inputs on the left are just counts of cells.}

\section{Conclusion}

We presented an intrinsic, measurement-first geometry in which length is a count and curvature is the operational coefficient in the small–circle/small–ball mismatch observed with a one length unit that equals one cell crossing. Distances are defined purely by shortest face-crossing counts, and this count metric is geodesic on every locally finite complex. The resulting excess-radius diagnostics are assembled into unified small-ball/small-sphere estimators that (i) vanish on undeformed lattices, (ii) detect local contraction/dilation with the correct sign, and (iii) identify the smooth scalar curvature in the continuum limit.

Crucially, the scalar estimator $K_{h,R}^{(m)}(x;r)$, the directional sectional estimators $\widehat{K}_{ij,a}(x;r)$, and their Ricci/trace combinations $\widehat{\mathrm{Ric}}_a$ and $\widehat{R}_a$ are obtained from \emph{counts only}: one fixes a base cell, chooses an intrinsic radius, counts the cells in the corresponding intrinsic ball (or tube-restricted slice), and compares the measured radius to the reconstructed radius in the same gauge. No coordinates, no angles, no derivatives, and no shape assumptions are required to perform the measurement. The smooth metric $g=e^{2u}g_0$ enters only to interpret the limiting object and to prove error bounds: under mild regularity and mesh hypotheses we obtain uniform rates $O(a/r + r)$, and in the joint limit $a\to0$, $r\to0$, $a/r\to0$ the discrete estimators converge to the Riemannian scalar curvature $R_g$, the sectional curvatures $K_g(e_i\wedge e_j)$, and hence to $\mathrm{Ric}_g$ itself.

The framework is therefore micro-agnostic and angle-free, yet admits a precise continuum interface. It provides a direct operational bridge from raw counts on a deforming cellular medium to standard curvature data of a smooth conformal relaxation $g=e^{2u}g_0$, with quantified stability (via small-ball rates) and convergence in the measured Gromov–Hausdorff sense.
\


\appendix
\section{Baseline counts and calibrations}
\label{app:counts}

This appendix fixes notation and removes the potential ambiguity between
(i) the \emph{adjacency–specific} lattice calibrations used for discrete diagnostics,
and (ii) the \emph{Euclidean} calibration used in the $R$–normalized estimator
(cf.\ Sec.~\ref{sec:local-global-interpretation}).

\section{Adjacency–specific baseline (discrete diagnostics)}
For the $\ell^1$ adjacency on $\mathbb{Z}^m$, the intrinsic ball and sphere counts are
\begin{equation}\label{eq:L1counts}
|B_m(r)| \;=\; \sum_{k=0}^{m} 2^{k}\binom{m}{k}\binom{r}{k},
\qquad
|S_m(r)| \;=\; |B_m(r)|-|B_m(r-1)|
\end{equation}
with the convention $\binom{r}{0}=1$ and $\binom{r}{k}=0$ for $k>r$.
In particular:
\[
\begin{aligned}
m=2:&\quad |S_2(r)|=4r,\quad |B_2(r)|=2r^2+2r+1\\
m=3:&\quad |B_3(r)|=\tfrac{4}{3}r^3+2r^2+\tfrac{8}{3}r+1
\end{aligned}
\]
and $|S_3(r)|=|B_3(r)|-|B_3(r-1)|=4r^2+2$.
Thus the undeformed lattice calibrations we use for \emph{discrete} radius reconstruction are
\[
\alpha_1^{\square}=4,\qquad
\beta_2^{\square}=2,\qquad
\beta_3^{\square}=\tfrac{4}{3},\qquad
\beta_4^{\square}=\tfrac{2}{3}
\]
(and analogously for hexagonal/cubic/hypercubic variants), so that
$r_c^{(1)}=C_h/\alpha_1$, $r_c^{(2)}=\sqrt{A_h/\beta_2}$, $r_c^{(3)}=\sqrt[3]{V_h/\beta_3}$, etc.

\begin{lemma}[Derivation of \eqref{eq:L1counts}]
\label{lem:L1counts}
For $m\ge1$ and $r\in\mathbb{N}$,
$|B_m(r)|=\sum_{k=0}^{m}2^{k}\binom{m}{k}\binom{r}{k}$.
\end{lemma}

\begin{proof}
An integer vector $x\in\mathbb{Z}^m$ satisfies $\|x\|_1\le r$ iff exactly $k$ coordinates are nonzero
($0\le k\le \min\{m,r\}$), we choose which $k$ coordinates in $\binom{m}{k}$ ways,
choose their signs in $2^k$ ways, and distribute a sum $\le r$ among $k$ \emph{positive} integers:
the number of compositions of an integer $\le r$ into $k$ parts (allowing $0$) equals $\binom{r}{k}$.
Summing over $k$ yields the formula.
$|S_m(r)|$ is the forward difference $|B_m(r)|-|B_m(r-1)|$.
\end{proof}

\section{Euclidean calibration (for the $R$–normalized estimator)}
When we identify the small–ball mismatch with the \emph{smooth} scalar curvature,
we normalize against the Euclidean volume $\alpha_m=\omega_m$:
\begin{equation*}
\mathrm{Vol}_g\big(B_g(x,r)\big)
= \alpha_m r^m - \alpha_m\frac{R_g(x)}{6(m+2)}r^{m+2}+O(r^{m+3})
\end{equation*}
and therefore set \(\beta_m=\alpha_m=\omega_m\) in the definition of the
$R$–normalized estimator
\(
K_{h,R}^{(m)}(x;r)=6(m+2)\big(r^{m}-(r_c^{(m)})^{m}\big)/r^{m+2}
\),
so that \(K_{h,R}^{(m)}(x;r)\to R_g(x)\) as \(r\downarrow0\) (Theorem~\ref{thm:R-local}).
\medskip

\noindent\emph{Summary.} Use $(\alpha_1^{\square},\beta_m^{\square})$ when comparing against the
undeformed \emph{adjacency}, and use $\beta_m=\omega_m$ when targeting $R_g$.
\section{Examples on Voronoi complexes}
\label{app:voronoi}

Let $\Omega\subset\mathbb{R}^m$ be open, $u\in C^2(\Omega)$, and $g=e^{2u}g_0$. 
For $h>0$ let $\mathcal{P}_h\subset\Omega$ be a finite set that is \emph{quasi-uniform in $g$}:
there exist constants $0<\delta\le \Delta<\infty$ (independent of $h$) such that
\begin{equation}\label{eq:qu}
\delta h \ \le\ \mathrm{dist}_g(p,\mathcal{P}_h\!\setminus\!\{p\}) \ \le\ \sup_{x\in\Omega}\mathrm{dist}_g(x,\mathcal{P}_h)\ \le\ \Delta h
\quad \text{for all }p\in\mathcal{P}_h
\end{equation}
Let $\{V_p\}_{p\in\mathcal{P}_h}$ be the \emph{Voronoi tessellation} in the $g$–metric, and let $X_h$ be the cell complex whose cells are $\{V_p\}$ with face adjacency.
Define $\Phi_h(V_p):=p$ and the count metric $d_h$ on $X_h$ as in Sec.~\ref{sec:metric}. 
Let $\mu_h$ assign to each cell its $g$–volume, i.e. $\mu_h(\{V_p\})=\mathrm{Vol}_g(V_p)$, and set $\nu_h:=\mu_h$ (so $\nu_h$ is volumetric).

\begin{proposition}[Voronoi complexes satisfy (H1)–(H5)]
\label{prop:voronoi-H}
Assume \eqref{eq:qu} and that $u$ is $C^2$ with bounded derivatives on a compact $K\Subset\Omega$. Then, for $h$ small enough, the Voronoi complex $X_h$ with realization $\Phi_h$ and measure $\nu_h$ satisfies \textup{(H1)}–\textup{(H5)} on $K$ with $a:=h$.
\end{proposition}

\begin{proof}
Fix $K\Subset\Omega$. Bounds in \eqref{eq:qu} imply $g$–balls of radius $\asymp h$ contain and are contained in each Voronoi cell intersecting $K$; since $u$ is $C^2$ on $K$, $g$ and $g_0$ are bi-Lipschitz equivalent on $K$ with constants independent of $h$. We verify the hypotheses.

\emph{(H1) Local finiteness / degree bound.}
If cells have inner/outer $g$–radii bounded by $c_1 h$ and $c_2 h$ uniformly on $K$ (a consequence of \eqref{eq:qu}), then only $O(1)$ cells can meet a given $V_p$; in particular, the face-degree is uniformly bounded, so the adjacency graph is locally finite.

\emph{(H2) Shape regularity / mesh size.}
For $p$ with $V_p\cap K\neq\emptyset$, quasi-uniformity yields
\[
B_g\!\big(p,c_1 h\big)\ \subset\ V_p\ \subset\ B_g\!\big(p,c_2 h\big)
\]
hence $\mathrm{diam}_{g_0}(V_p)\in [\tilde c_1 h,\tilde c_2 h]$ on $K$ by bi-Lipschitz equivalence of $g$ and $g_0$.

\emph{(H3) Coarse realization.}
$\Phi_h$ is the identity on sites; the $g_0$–distance between adjacent cell centers is $\le C h$ by the outer radius bound, hence the coarse realization and $Ch$–density on $K$ follow.

\emph{(H4) Ball inclusions.}
Let $x\in X_h$ with $\Phi_h(x)\in K$. Any $g$–path of length $L$ from $\Phi_h(x)$ intersects at most $C\, L/h$ distinct Voronoi cells, because each traversal across a cell consumes at least the $g$–inradius $\ge c_1 h$. Therefore $d_h(x,y)\le C\, d_g(\Phi_h(x),\Phi_h(y))/h$. Conversely, a shortest $d_h$–path of $r$ cells has $g$–length $\ge c' r h$ since each step crosses a face at $g$–distance $\ge c' h$. Thus
\[
B_g\big(\Phi_h(x),\,c' r h\big)\ \subseteq\ \Phi_h\big(B_{d_h}(x,r)\big)\ \subseteq\ B_g\big(\Phi_h(x),\,C r h\big)
\]
which is the desired inclusion with $a=h$.

\emph{(H5) Weight comparability.}
On $K$, $u$ is bounded $C^2$, hence $\mathrm{Vol}_g(V_p)\asymp h^m e^{m u(p)}$ uniformly over cells meeting $K$. Thus there is $\Lambda\ge1$ with 
\[
\Lambda^{-1} h^{m} e^{m u(\Phi_h(V_p))}\ \le\ \nu_h(\{V_p\})\ \le\ \Lambda h^{m} e^{m u(\Phi_h(V_p))}
\]
All constants depend only on $K$, $\delta,\Delta$, and bounds on $u$ and its derivatives on $K$, not on $h$.
\end{proof}

\begin{corollary}[Directional and scalar convergence on Voronoi]
\label{cor:voronoi-conv}
Under the assumptions of Proposition~\ref{prop:voronoi-H}, let $\widehat{K}_{ij,h}(x;r)$ be the directional estimator computed from counts in an $O(h)$–thick Fermi tube around the geodesic slice $\Sigma_{ij}$ 
(as in Lemma~\ref{lem:slice-coupling}), and let $\widehat{R}_m$ be the $R$–normalized small-ball estimator \eqref{eq:Rnorm-sec4}.
Then for all $x$ with $\Phi_h(x)\in K$ and all admissible $r$,
\[
\begin{aligned}
\bigl|\widehat{K}_{ij,h}(x;r)-K_g(e_i\wedge e_j)(x)\bigr|
&\le C\!\left(\frac{h}{r}+r\right)\\
\bigl|\widehat{R}_m(x;r)-R_g(\Phi_h(x))\bigr|
&\le C\!\left(\frac{h}{r}+r\right)
\end{aligned}
\]

and choosing $r\simeq \sqrt{h}$ yields error $O(\sqrt{h})$. The constants depend only on $K$, the quasi-uniformity parameters, and $C^2$ bounds for $u$ on $K$.
\end{corollary}

\begin{proof}
Proposition~\ref{prop:voronoi-H} verifies \textup{(H1)}–\textup{(H5)} with $a=h$. The bounds then follow from Theorem~\ref{thm:rates} (scalar) and Lemma~\ref{lem:slice-coupling} with its corollary (directional), observing that the same $O(h/r+r)$ rate holds uniformly on $K$. The minimizer of $h/r+r$ is $r=\sqrt{h}$, giving $O(\sqrt{h})$.
\end{proof}

\section{Worked toy: conformal bump via density}
\label{app:toy}

For $m\in\{2,3\}$, let $\rho(x)=\exp(\varepsilon\varphi(x))$ with $\varphi$ smooth, compactly supported, and $\varepsilon$ small. Set $u=\tfrac{1}{m}\log\rho=\tfrac{\varepsilon}{m}\varphi$. 
The scalar curvature signal from density alone (Sec.~\ref{sec:counts-curvature}) is
\[
R(x)=\rho(x)^{-2/m}\!\left[-\frac{2(m-1)}{m}\frac{\Delta\rho}{\rho}+\frac{(m-1)(m+2)}{m^2}\frac{|\nabla\rho|^2}{\rho^2}\right]
\]
To first order in $\varepsilon$ this reduces to
\(
R(x) = -\tfrac{2(m-1)}{m^2}\,\varepsilon\,\Delta\varphi(x) + O(\varepsilon^2).
\)
Sampling a quasi-uniform $\mathcal{P}_h$ with respect to $g=e^{2u}g_0$ and forming the Voronoi complex $X_h$, 
Corollary~\ref{cor:voronoi-conv} with $r\simeq \sqrt{h}$ yields
\[
\widehat{R}_m(x;r)\ =\ R(x)\ +\ O(\sqrt{h})\ +\ O(\varepsilon^2)
\]
uniformly on compact sets inside the support of $\varphi$.
This provides a crisp, non-lattice sanity check for the sign and localization of the estimator.
\section{Proof of the measured GH limit (Theorem~\ref{thm:mgh})}
\label{app:mgh-proof}

We prove convergence on compacts and then exhaust $\Omega$. Throughout, $C$ denotes a constant depending only on the data in (H1)–(H5) on the compact set under consideration (and may change from line to line).

\begin{proposition}[Local mGH on compacts]
\label{prop:mgh-compact}
Fix $K\Subset\Omega$. Under \textup{(H1)–(H5)} on $K$, there exist $a_0>0$ and $C_K>0$ such that for all $a\in(0,a_0]$ there is a correspondence $\mathcal{R}_a\subset X_a\times\Omega$ with
\begin{equation}\label{eq:A1}
\mathrm{dis}(\mathcal{R}_a)\ \le\ C_K\,a
\end{equation}
and
\begin{equation}\label{eq:A2}
d_{\mathrm{P}}\!\Big((\Phi_a)_{\#}\nu_a^{K},\,\mathrm{vol}_g^{K}\Big)\ \le\ C_K\,a
\end{equation}
where $d_{\mathrm{P}}$ is the Prokhorov distance on \emph{probability} measures, and
\begin{equation}\label{eq:A3}
\nu_a^{K}\ :=\ \frac{\nu_a\!\llcorner \Phi_a^{-1}(K)}{\nu_a\big(\Phi_a^{-1}(K)\big)},\qquad
\mathrm{vol}_g^{K}\ :=\ \frac{\mathrm{vol}_g\!\llcorner K}{\mathrm{vol}_g(K)}
\end{equation}
Hence $(X_a,d_a^\ast,\nu_a)\to(\Omega,d_g,\mathrm{vol}_g)$ in the measured GH sense on $K$.
\end{proposition}

\begin{proof}
\noindent\textbf{Metric distortion bound.}
Choose $K'\Subset K$ with $\mathrm{dist}_g(K',\partial K)\ge c\,a$ for some fixed $c>0$. Set
\begin{equation}\label{eq:A4}
\mathcal{R}_a\ :=\ \big\{(x,\Phi_a(x)):\ \Phi_a(x)\in K'\big\}\ \subset X_a\times\Omega .
\end{equation}
By (H4), for any $x\in X_a$ with $\Phi_a(x)\in K'$ and any integer $r\ge0$,
\begin{equation}\label{eq:A5}
B_g\!\big(\Phi_a(x),\,ar-c_1 a\big)\ \subseteq\ \Phi_a\!\big(B_{d_h}(x,r)\big)\ \subseteq\
B_g\!\big(\Phi_a(x),\,ar+c_2 a\big)
\end{equation}
Let $x,y\in X_a$ with $\Phi_a(x),\Phi_a(y)\in K'$. Taking the minimal $r=d_h(x,y)$ in \eqref{eq:A5} and comparing the infimum $s$ such that $\Phi_a(y)\in B_g(\Phi_a(x),as+c_2 a)$ yields
\begin{equation}\label{eq:A6}
\Big|\, d_g\!\big(\Phi_a(x),\Phi_a(y)\big)\ -\ a\,d_h(x,y)\,\Big|\ \le\ C a 
\end{equation}
Since $d_a^\ast=a\,d_h$, \eqref{eq:A6} implies
\begin{equation}\label{eq:A7}
\sup_{(x,p),(y,q)\in\mathcal{R}_a}\ \big|\, d_g(p,q)-d_a^\ast(x,y)\,\big|\ \le\ C a
\end{equation}
i.e. $\mathrm{dis}(\mathcal{R}_a)\le C a$, proving \eqref{eq:A1} (with $C_K$ in place of $C$).

\noindent\textbf{Pushforward measure bound.}
Let $\psi\in C_c^\infty(K')$ with $\|\psi\|_{C^1(K)}\le 1$. Using (H2)–(H5) and a per–cell Taylor estimate,
\begin{align}
\int \psi\, d\big((\Phi_a)_{\#}\nu_a\big)
&=\ \sum_{\substack{c\in\mathcal{C}_a\\ \Phi_a(c)\in K'}} \psi\!\big(\Phi_a(c)\big)\,\nu_a(\{c\})\nonumber\\
&=\ \int_{K'} \psi\, d\mathrm{vol}_g\ +\ O_K(a) \label{eq:A8}
\end{align}
(Here the $O_K(a)$ depends only on shape/degree bounds and $\|u\|_{C^2(K)}$.) After normalizing to probabilities as in \eqref{eq:A3}, the same estimate gives
\begin{equation}\label{eq:A9}
\Big|\ \int \psi\, d\big((\Phi_a)_{\#}\nu_a^{K}\big)\ -\ \int \psi\, d\mathrm{vol}_g^{K}\ \Big|\ \le\ C_K\,a 
\end{equation}
Since $C_c^\infty(K')$ is dense in the bounded–Lipschitz class on $K'$ and both measures are tight on $K$, \eqref{eq:A9} implies
\begin{equation}\label{eq:A10}
d_{\mathrm{BL}}\!\Big((\Phi_a)_{\#}\nu_a^{K},\,\mathrm{vol}_g^{K}\Big)\ \le\ C_K\,a 
\end{equation}
The Prokhorov distance metrizes weak convergence on a compact metric space and is controlled by $d_{\mathrm{BL}}$; hence
\begin{equation}\label{eq:A11}
d_{\mathrm{P}}\!\Big((\Phi_a)_{\#}\nu_a^{K},\,\mathrm{vol}_g^{K}\Big)\ \le\ C_K\,a 
\end{equation}
which is \eqref{eq:A2}.

\noindent\textbf{Conclusion on $K$.}
Combining \eqref{eq:A1} and \eqref{eq:A2} with the standard characterization of measured GH convergence via correspondences with small distortion and small pushed–measure discrepancy yields the claim on $K$.
\end{proof}

\begin{proof}[Proof of Theorem~\ref{thm:mgh}]
Let $\{K_n\}_{n\ge1}$ exhaust $\Omega$ by compacts with smooth boundary. By Proposition~\ref{prop:mgh-compact}, for each fixed $n$,
\begin{equation}\label{eq:A12}
d_{\mathrm{mGH}}\!\Big((X_a,d_a^\ast,\nu_a)\big|_{K_n},\ (\Omega,d_g,\mathrm{vol}_g)\big|_{K_n}\Big)\ \le\ C_{K_n}\,a 
\end{equation}
A diagonal argument in $n$ gives measured GH convergence locally uniformly on compacts, hence $(X_a,d_a^\ast,\nu_a)\xrightarrow[\mathrm{mGH}]{}(\Omega,d_g,\mathrm{vol}_g)$.
\end{proof}

\begin{backmatter}
\bmsection{Funding} None.

\bmsection{Disclosures} The authors declare no conflicts of interest.

\bmsection{Data Availability} No datasets were generated or analyzed in this study.

\end{backmatter}

\nocite{*}
\bibliography{ref}

@misc{BarakHAL,
  author       = {Barak, Shlomo},
  title        = {Universe Geometry of a Deformed 3D Space Lattice},
  howpublished = {HAL archive},
  year         = {2025},
  note         = {HAL Id: hal-04670068},
  url          = {https://hal.science/hal-04670068}
}

@article{Regge1961,
  author  = {Regge, Tullio},
  title   = {General relativity without coordinates},
  journal = {Il Nuovo Cimento},
  volume  = {19},
  number  = {3},
  pages   = {558--571},
  year    = {1961}
}

@book{BuragoBuragoIvanov2001,
  author    = {Burago, Dmitri and Burago, Yuri and Ivanov, Sergei},
  title     = {A Course in Metric Geometry},
  series    = {Graduate Studies in Mathematics},
  volume    = {33},
  publisher = {American Mathematical Society},
  year      = {2001}
}

@book{Morgan1998,
  author    = {Morgan, Frank},
  title     = {Riemannian Geometry: A Beginner's Guide},
  edition   = {2},
  publisher = {A K Peters},
  year      = {1998}
}

@book{doCarmo1992,
  author    = {do Carmo, Manfredo P.},
  title     = {Riemannian Geometry},
  publisher = {Birkh{\"a}user},
  year      = {1992}
}

@book{BobenkoSuris2008,
  author    = {Bobenko, Alexander I. and Suris, Yuri B.},
  title     = {Discrete Differential Geometry: Integrable Structure},
  publisher = {American Mathematical Society},
  year      = {2008}
}

@book{Petersen2016,
  author    = {Petersen, Peter},
  title     = {Riemannian Geometry},
  edition   = {3},
  series    = {Graduate Texts in Mathematics},
  volume    = {171},
  publisher = {Springer},
  year      = {2016}
}

@book{Wald1984,
  author    = {Wald, Robert M.},
  title     = {General Relativity},
  publisher = {University of Chicago Press},
  year      = {1984}
}

@book{MTW1973,
  author    = {Misner, Charles W. and Thorne, Kip S. and Wheeler, John Archibald},
  title     = {Gravitation},
  publisher = {W. H. Freeman},
  year      = {1973}
}

@article{Einstein1933,
  author  = {Einstein, Albert},
  title   = {On the Method of Theoretical Physics},
  journal = {Philosophy of Science},
  volume  = {1},
  number  = {2},
  pages   = {163--169},
  year    = {1934},
  note    = {Herbert Spencer Lecture, Oxford, 1933}
}

@book{Rindler2004,
  author    = {Rindler, Wolfgang},
  title     = {Relativity: Special, General, and Cosmological},
  publisher = {Oxford University Press},
  year      = {2004}
}
\end{document}